\documentclass[12pt,a4paper,reqno]{amsart}
\usepackage{amsmath,amsfonts,amssymb,amsthm,amscd}
\usepackage[english]{babel}
\usepackage[mathscr]{eucal}
\usepackage{graphicx}
\usepackage{hyperref}

\renewcommand{\a}{\alpha}
\renewcommand{\b}{\beta}
\newcommand{\g}{\gamma}
\newcommand{\G}{\Gamma}
\renewcommand{\d}{\delta}
\newcommand{\D}{\Delta}
\newcommand{\h}{\chi}

\newcommand{\om}{\omega}
\newcommand{\OM}{\Omega}

\newcommand{\s}{\sigma}

\renewcommand{\t}{\tau}
\renewcommand{\th}{\theta}
\newcommand{\e}{\varepsilon}
\newcommand{\f}{\varphi}
\newcommand{\F}{\Phi}
\newcommand{\x}{\xi}

\newcommand{\y}{\eta}

\newcommand{\p}{\psi}

\newcommand{\C}{{\mathbb C}}

\newcommand{\N}{{\mathbb N}}
\newcommand{\R}{{\mathbb R}}

\newcommand{\RR}{{\mathbb R}^2}

\newcommand{\curl}{{\rm curl}\,}

\newcommand{\diver}{{\rm div}\,}
\newcommand{\inte}{{\rm int}\,}
\newcommand{\pd}{\partial}

\newcommand{\dt}{\partial_t}

\newcommand{\supp}{\operatorname{supp\,}}

\newcommand{\loc}{\operatorname{{loc}}}

\newtheorem{theorem}{Theorem}[section]
\newtheorem{proposition}[theorem]{Proposition}
\newtheorem{lemma}[theorem]{Lemma}
\newtheorem{corollary}[theorem]{Corollary}

\newtheorem{definition}[theorem]{Definition}
\newtheorem{assumption}[theorem]{Assumption}

\theoremstyle{remark}
\newtheorem{remark}[theorem]{Remark}
\newtheorem{example}[theorem]{Example}

\numberwithin{equation}{section}

\newcommand{\dss}{{\mathbf{ds}}}
\newcommand{\na}{\nabla}

\begin{document}


\title[Small curve]{Two Dimensional Incompressible Ideal Flow Around a Small Curve}
\author{C. Lacave}

\address[C. Lacave]{Universit\'e Paris-Diderot (Paris 7)\\
Institut de Math\'ematiques de Jussieu\\
UMR 7586 - CNRS\\
175 rue du Chevaleret\\
75013 Paris\\
France} \email{lacave@math.jussieu.fr}

\date{\today}

\begin{abstract}
We study the asymptotic behavior of solutions of the two dimensional incompressible Euler equations in the exterior of a curve when the curve shrinks to a point. This work links two previous results:  [Iftimie, Lopes Filho and Nussenzveig Lopes, Two Dimensional Incompressible Ideal Flow Around a Small Obstacle, Comm. PDE, {\bf 28} (2003), 349-379] and [Lacave, Two Dimensional Incompressible Ideal Flow Around a Thin Obstacle Tending to a Curve, Ann. IHP, Anl \textbf{26} (2009), 1121-1148]. The second goal of this work is to complete the previous article, in defining the way the obstacles shrink to a curve. In particular, we give geometric properties  for domain convergences in order that the limit flow be a solution of Euler equations.
\end{abstract}

\maketitle

\section{Introduction}

The purpose of this work is to study the influence of a material curve on the behavior of two-dimensional ideal flows when the size of the curve tends to zero. The study of the fluid flows in a singularly perturbed domains was initiated by Iftimie, Lopes Filho and Nussenzveig Lopes in \cite{ift_lop_euler}, in the case of a smooth obstacle which shrinks to a point. For some initial data, they obtain a blow-up of the limit velocity like $1/|x|$ centered at the point where the obstacle disappears. For some other initial data, they prove that there is no blow-up. Six years later, the case of thin obstacles shrinking to a curve was treated in \cite{lac_euler}. It was shown that the limit velocity always blows up at the end-points of the curve like $1/\sqrt{|x|}$. In light of this two works a natural question arises: what happens in the case of small curves ? Our result can be stated as following: as the end-points get closer and closer, for some initial data, the two blow-ups like $1/\sqrt{|x|}$ combine in order to give $1/|x|$, and for other initial data, the blow-ups compensate each other and disappear.

More precisely, we fix both an initial vorticity $\om_0$, smooth and compactly supported outside the obstacle $\OM$, and the circulation $\g$ of the initial velocity around the obstacle.
We assume that the obstacle $\OM$ is a bounded, connected, simply connected subset of the plane. Let us define the exterior domain $\Pi := \R^2 \setminus \overline{\OM}$. Then, the vorticity and the circulation uniquely determine a vector field $u_0$ tangent to the obstacle such that:
\[ \diver u_0 = 0, \ \curl u_0 = \om_0, \ \lim_{|x|\to \infty} u_0(x) = 0,\ \oint_{\pd \OM} u_0\cdot \dss = \g .\]

When the obstacle $\OM$ is smooth and open, it is proved by Kikuchi \cite{kiku} that there exists a unique global strong solution to the Euler equations in $\Pi$. If $\OM$ is a smooth curve $\G$ (with two end-points), we have to define what is a weak-solution.

\begin{definition}\label{sol-curve}
Let $\om_0\in L^1\cap L^\infty(\RR)$ and $\g\in\R$. We say that $(u,\om)$ is a global weak solution of the Euler equations outside the curve $\G$ with initial condition $(\om_0,\g)$ if
\begin{equation*}
\om\in L^\infty(\R^+;L^1\cap L^\infty(\RR))
\end{equation*}
and if we have in the sense of distributions
\begin{equation}\label{transport}
\begin{cases}
 \dt \om +\diver(u \om)=0, \\
 \om(0)=\om_0,
\end{cases}
\end{equation}
where $u$ verifies
\begin{equation*}
\left\lbrace\begin{aligned}
\diver u &=0 &\text{ in } {\Pi} \\
\curl u &=\om &\text{ in } {\Pi} \\
u\cdot\hat{n}&=0 &\text{ on } \G \\
\oint_{\G} u \cdot \dss& = \g & \text{ for }t\in[0,\infty)\\
\lim_{|x|\to\infty}|u|&=0.
\end{aligned}\right.
\end{equation*}
\end{definition}
In this definition, $\Pi:= \R^2\setminus \G$, and \eqref{transport} means that we have 
\[ \int_0^\infty\int_{\R^2}\f_t\om dxdt +\int_0^\infty \int_{\R^2}\na\f.u\om dxdt+\int_{\R^2}\f(0,x)\om_0(x)dx=0,\]
for any test function $\f\in C^\infty_c([0,\infty)\times\R^2)$.

In \cite{lac_euler}, we prove the existence of a global weak solution in the sense of the previous definition. The idea of the previous paper is the following: for $\G$ given, we manage to construct a sequence of smooth obstacles $\OM_n$ (thanks to biholomorphisms), which shrink to the curve. Next, we consider the strong solution $(u^n,\om^n)$ in smooth domain $\Pi_n := \R^2\setminus \overline{\OM_n}$, and we pass to the limit. The details of this proof will be presented in Subsection \ref{thicken}.

However, we have constructed a special family of obstacles. The first goal here is to generalized \cite{lac_euler} in the case of a geometrical convergence of $\OM_n$ to $\G$.

\begin{theorem}\label{main 2}
Let $\{\OM_n\}$ be a sequence of smooth, open obstacles containing $\G$. If $\OM_n \to \G$ in the sense of Theorem \ref{theo-assump}, then there exists a subsequence $n=n_k\to 0$ such that
\begin{itemize}
\item[(a)] $\F^n u^n \to u$ strongly in $L^2_{\loc}(\R^+\times\R^2)$;
\item[(b)] $\F^n \om^n \to \om$ weak-$*$ in $L^\infty(\R^+;L^4(\R^2))$;
\item[(c)] $(u,\om)$ is a global weak solution of the Euler equations around the curve $\G$.
\end{itemize}
\end{theorem}
In this result, $\F^n$ is a cut-off function of an $\frac1n$-neighborhood of $\OM_n$. In particular, we will remark that this sense of convergence holds for smooth domain, i.e. if $\OM_n\to \OM$, where $\OM$ is smooth, then the Euler solutions on $\OM_n$ (respectively on the exterior domain $\Pi_n$) tends to the Euler solution on $\OM$ (respectively on $\Pi$). Another consequence of proving a geometrical theorem (Theorem \ref{theo-assump}) is the extension of \cite{lac_NS}, which corresponds to the previous theorem in the viscous case (with Navier-Stokes equations instead to Euler equations).

\bigskip

The second goal of this article is to study the behavior of the weak solution when the curve shrinks to a point.

In \cite{ift_lop_euler}, the authors fix a smooth obstacle $\OM_0$,  containing the origin, and choose $\OM_\e := \e \OM_0$. For this homothetic convergence, they prove the following theorem.
\begin{theorem} \label{ift_main}
There exists a subsequence $\e=\e_k\to 0$ such that
\begin{itemize}
\item[(a)] $\F^\e u^\e \to u$ strongly in $L^1_{\loc}(\R^+\times\R^2)$;
\item[(b)] $\F^\e {\om}^\e \to {\om}$ weak $*$ in $L^\infty(\R^+\times\R^2)$;
\item[(c)] the limit pair $(u,{\om})$ verify in the sense of  distributions: 
\begin{equation*}
\left\lbrace\begin{aligned}
&\pd_t {\om}+u\cdot \na {\om}=0 &\text{ in } (0,\infty) \times \R^2  \\
&\diver u =0 &\text{ in } (0,\infty) \times \R^2  \\
&\curl u ={\om} + \g \d_0 &\text{ in } (0,\infty) \times \R^2  \\
&\lim_{|x|\to\infty}|u|=0 & \text{ for }t\in[0,\infty)\\
&{\om}(0,x)={\om}_0(x) &\text{ in } \R^2
\end{aligned}\right.
\end{equation*}
with $\d_0$ the Dirac function at $0$.
\end{itemize}
\end{theorem}
In this result, $\F^\e$ is a cut-off function of an $\e$-neighborhood of $\OM_\e$.  Therefore, they obtain at the limit the Euler equations in the full plane, where a Dirac mass at the origin appears. This additional term is a reminiscense of the circulation $\g$ of the initial velocities around the obstacles, and we note that this term does not appear if  $\g=0$. Actually, we can write the velocity as a sum of a smooth vector field and $\g \dfrac{x^\perp}{2\pi |x|^2}$. Additionally, \cite{lac_miot} proves that there exists at most one global solution of the previous limit system. Therefore, we can state that Theorem \ref{ift_main} holds true for all sequences $\e_k\to 0$, without extracting subsequences.

In the exterior of the curve, we will note in Remark \ref{rem : prop-curve} that the velocity, for any weak solution, is continuous up to the curve, with different values on each side of $\G$, and blows up at the endpoints of the curve as the inverse of the square root of the distance. As it was said at the beginning of this introduction, we remark in this two results that the velocity blows up like $1/|x|$ in the case of a point, and like $1/\sqrt{|x|}$ near the end-points in the case of the curve. The problem here is to check that we find a result similar to Theorem \ref{ift_main} when a curve shrinks to a point. 

We fix a smooth open Jordan arc $\G$, and we set $\G_\e := \e \G$. Then there exists at least one weak solution of Euler equation outside the curve $\G_\e$. Our goal is to prove the following theorem.
\begin{theorem} \label{main}
Let $(u^\e,\om^\e)$ be a weak solution for Euler equation outside $\G_\e$. Then, for all sequences $\e=\e_k\to 0$, we have
\begin{itemize}
\item[(a)] $u^\e \to u$ strongly in $L^1_{\loc}(\R^+ \times \R^2)$;
\item[(b)] ${\om}^\e \to {\om}$ weak $*$ in $L^\infty(\R^+ \times \R^2)$;
\item[(c)] the limit pair $(u,{\om})$ is the unique solution in the sense of  distributions of:
\begin{equation*}
\left\lbrace\begin{aligned}
&\pd_t {\om}+u\cdot \na {\om}=0 &\text{ in } (0,\infty) \times \R^2  \\
&\diver u =0 &\text{ in } (0,\infty) \times \R^2  \\
&\curl u ={\om} + \g \d_0 &\text{ in } (0,\infty) \times \R^2  \\
&\lim_{|x|\to\infty}|u|=0 & \text{ for }t\in[0,\infty)\\
&{\om}(0,x)={\om}_0(x) &\text{ in } \R^2.
\end{aligned}\right.
\end{equation*}
\end{itemize}
\end{theorem}

Although Theorems \ref{ift_main} and \ref{main} appear to be similar, the estimations and arguments are different. In all these works, the main tool is the explicit formula of the Biot-Savart law (law giving the velocity in terms of vorticity), thanks to conformal mappings. In \cite{ift_lop_euler}, the authors use the change of variables $y=x/\e$, in order to work in a fixed domain. Then they obtain $L^\infty$ estimates for a part of the velocity, thanks to the smoothness of the obstacle, and they pass to the limit using a Div-Curl Lemma. In \cite{lac_euler}, the change of variables does not hold and we work on convergence of biholomorphisms when the domains change. Next, we take advantage of this convergence to pass directly to the limit in the Biot-Savart law, thanks to the dominated convergence theorem. In our case of the small curve, we lost the convergence of biholomorphism and \cite{lac_euler} cannot be applied directly. The curve being a non-smooth domain, because of the end-points, we cannot either apply directly the result from \cite{ift_lop_euler}. Indeed, we will see that we only have $L^p$ estimates of the velocity for $p<4$ instead of $L^\infty$. Actually, we will improve some estimates and we will manage to pass to the limit with the Div-Curl Lemma.

The remainder of this work is organized in three sections. In Section \ref{sect : 2}, we recall some results on conformal mapping and Biot-Savart law. We show that \cite{lac_euler} gives us the existence of at least one global weak solution $(u^\e,\om^\e)$ for the Euler equations outside the curve $\G_\e$ (in the sense of Definition \ref{sol-curve}). We take advantage of this part to prove the convergence of biholomorphisms when an obstacle shrinks  to a curve, completing \cite{lac_euler,lac_NS}. Theorem \ref{main 2} will be a consequence of this section. There will be a general remark concerning the convergence of Euler solutions when the domain converges. In Section \ref{sect : 3}, we establish {\it a priori} estimates for the vorticity and the velocity, in order to pass to the limit in the last section.

We emphasize that the techniques used here, and in \cite{ift_lop_euler, lac_euler}, are specific to the ideal flows in dimension two, around one obstacle. The study of several obstacles does not allow us to use Riemann mappings. Loosing the explicit formula of the Biot-Savart law, the author in \cite{milton} choose to work in a bounded domain, with several holes, where one hole shrinks to a point. In a bounded domain, he can use the maximum principle, and he obtains a theorem similar to Theorem \ref{ift_main}. In \cite{lac_lop}, we work in an unbounded domain, with $n$ obstacles, of size $\e$, uniformly distributed on a imaginary curve $\G$, and the goal is to determine if a chain of small islands ($n\to \infty$, $\e\to 0$) has the same effect as a wall $\G$. 

Concerning the viscous flow, there is no control of the vorticity in a domain with some boundary. Therefore, we do not use $\om$ and the Biot-Savart law, and we gain control on one derivative of the velocity thanks to the energy inequality. The case of a small obstacle in dimension two is studied in \cite{ift_lop_NS}, and in dimension three in \cite{ift_kell}. The thin obstacle in dimension two is treated in \cite{lac_NS}, and in dimension three in \cite{lac_3D}. We note that there does not exist any result for ideal flow in dimension three. Indeed, we cannot use the vorticity equation, and we do not have energy inequality. Then, a control of  the derivative of the velocity is missing.

For the sake of clarity, the main notations are listed in an appendix at the end of the paper.

\section{Conformal mapping}\label{sect : 2}

As it is mentioned in the introduction, complex analysis is an important tool for the study of two dimensional ideal flow outside one obstacle. Identifying $\R^2$ with the complex plane $\C$, the biholomorphism mapping the exterior of the obstacle to the exterior of the unit disk will be used to obtain an explicit formula for the Biot-Savart law. A key of this work, as in \cite{ift_lop_euler,lac_euler}, is to estimate these biholomorphisms when the size and the sickness of the obstacle go to zero. We begin this section by some reminders on thin obstacles (see \cite{lac_euler} for more details).

\subsection{Thin obstacles}\

Let $D=B(0,1)$ and $S=\pd D$.

We begin by giving some basic definitions on the curve.

\begin{definition} We call a {\it Jordan arc} a curve $C$ given by a parametric representation $C:\f(t)$, $0\leq t\leq 1$ with $\f$ an injective ($=$ one-to-one) function, continuous on $[0,1]$. An {\it open Jordan arc} has a parametrization $C:\f(t)$, $0< t<1$ with $\f$ continuous and injective on $(0,1)$.

We call a {\it Jordan curve} a curve $C$ given by a parametric representation $C:\p(t)$, $t\in \R$, $1$-periodic, with $\p$ an injective function on $[0,1)$, continuous on $\R$.
\end{definition}

Thus a Jordan curve is closed ($\f(0)=\f(1)$) whereas a Jordan arc has distinct endpoints. If $J$ is a Jordan curve in $\C$, then the Jordan Curve Theorem states that $\C\setminus J$ has exactly two components $G_0$ and $G_1$, and these satisfy $\pd G_0=\pd G_1=J$.

The Jordan arc (or curve) is of class $C^{n,\a}$ ($n\in\N^*,0<\a\leq 1$) if its parametrization $\f$ is $n$ times continuously differentiable, satisfying $\f'(t)\neq 0$ for all $t$, and if $|\f^{(n)}(t_1)-\f^{(n)}(t_2)|\leq C|t_1-t_2|^\a$ for all $t_1$ and $t_2$.

Let $\G: \G(t),0\leq t\leq 1$ be a Jordan arc. Then the subset $\R^2\setminus\G$ is connected and we will denote it by $\Pi$. The purpose of the following proposition is to obtain some properties of a biholomorphism $T: \Pi \to \inte\ D^c$. After applying a homothetic transformation, a rotation and a translation, we can suppose that the endpoints of the curve are $-1=\G(0)$ and $1=\G(1)$.

\begin{proposition}\label{2.2}
If $\G$ is a $C^2$ Jordan arc, such that the intersection with the segment $[-1,1]$ is a finite union of segments and points, then there exists a biholomorphism $T:\Pi\to \inte\ D^c$ which verifies the following properties:
\begin{itemize}
\item $T^{-1}$ and $DT^{-1}$ extend continuously up to the boundary, and $T^{-1}$ maps $S$ to $\G$,
\item $DT^{-1}$ is bounded,
\item $T$ and $DT$ extend continuously up to $\G$ with different values on each side of $\G$, except at the endpoints of the curve where $T$ behaves like the square root of the distance and $DT$ behaves like the inverse of the square root of the distance,
\item $DT$ is bounded in the exterior of the disk $B(0,R)$, with $\G \subset B(0,R)$,
\item $DT$ is bounded in $L^p(\Pi\cap B(0,R))$ for all $p<4$ and $R>0$.
\end{itemize}
\end{proposition}

The behavior of $T$ and $DT$ gives us the behavior of the velocity around the curve (see Proposition \ref{remark_lac_euler}). We rewrite also a remark from \cite{lac_euler} concerning behavior of biholomorphisms at infinity.

\begin{remark} \label{2.5}If we have a biholomorphism $H$ between the exterior of an open connected and simply connected domain $A$ and $D^c$, such that $H(\infty)=\infty$, then there exists a nonzero real number $\b$ and a  holomorphic function $h:\Pi\to \C$ such that:
$$H(z)=\b z+ h(z).$$
with 
$$h'(z)=O\Bigl(\frac{1}{|z|^2}\Bigl), \text{ as }|z|\to\infty.$$

This property can be applied for the $T$ above.
\end{remark}

In \cite{lac_euler,lac_NS}, we consider a family of obstacles $\{\OM_\y \}$ which shrink to the curve $\G$ in the following sense. If we denote by $T_\y$ the biholomorphism between $\Pi_\y:= \R^2\setminus \overline{\OM_\y}$ and $D^c$, then we supposed the following properties:

\begin{assumption} \label{assump}
The biholomorphism family   $\{T_\y\}$ verifies
\begin{itemize}
\item[(i)] $\|(T_\y- T)/|T|\|_{L^\infty(\Pi_\y)}\to 0$ as $\y \to 0$,
\item[(ii)] $\det(DT_\y^{-1})$ is bounded on $D^c$ independently of $\y$,
\item[(iii)] for any $R>0$, $\|DT_\y - DT\|_{L^3(B(0,R)\cap \Pi_\y)}\to 0$ as $\y \to 0$,
\item[(iv)] for $R>0$ large enough, there exists  $C_R>0$ such that $|DT_\y(x)|\leq C_R$ on $B(0,R)^c$.
\item[(v)]  for $R>0$ large enough, there exists  $C_R>0$ such that $|D^2 T_\y(x)|\leq \frac{C_R}{|x|}$ on $B(0,R)^c$.
\end{itemize}
\end{assumption}

\begin{remark}
We can observe that property (iii) implies that for any $R$, $DT_\y$ is bounded in $L^p(B(0,R)\cap \Pi_\y)$ independently of $\y$, for $p\leq3$. Moreover, condition (i) means that $T_\y\to T$ uniformly on $B(0,R)\cap \Pi_\y$ for any $R>0$.
\end{remark}

Assumption \ref{assump} corresponds to Assumption 3.1 in \cite{lac_euler}, adding part (v) and strengthening property (i) therein. 

Concerning our problem of the small curve, we will not need to assume something. Indeed, in the following subsection, we will present a way of thicken the curve such that all the properties of Assumption \ref{assump} are verified. However, an open problem raised by \cite {lac_euler,lac_NS} is to prove that Assumption \ref{assump} is also verified for more general geometrical convergences $\OM_\y \to \G$. It is the purpose of Subsection \ref{assump-proof}.

\subsection{Thicken the curve}\label{thicken}\

Thanks to Proposition \ref{2.2},  for a curve $\G$ given, we associate its biholomorphism $T$. Therefore, we have for
\begin{equation}\label{T_eps}
 \G_\e := \e \G,\ T_\e(x)=T(x/\e).
 \end{equation}

We recall that $T_\e$ maps $\Pi_\e:= \R^2\setminus \G_\e$ to $\R^2 \setminus \overline{D}$ and $\G_\e$ to $\pd D$. Let us define 
\[\OM_{\e,\y} := T_\e^{-1}(B(0,1+\y)\setminus D).\]
Knowing that $T_\e$ is a biholomorphism, we can state that $\OM_{\e,\y}$ is a smooth, bounded, open, connected, simply connected subset of the plane containing $\G_\e$. We can also remark that
\begin{equation}\label{T_eps_eta}
T_{\e,\y}(x)=\frac{1}{1+\y}T_\e(x)=\frac{1}{1+\y}T(x/\e)
\end{equation}
is a biholomorphism mapping $\Pi_{\e,\y}:= \OM_{\e,\y}^c$ to $D^c$ and $\pd \OM_{\e,\y}$ to $\pd D$.

\begin{example}
We give here the shape of $\OM_{\e,\y}$ in the special case of the segment. If $\G:= [(-1,0);(1,0)]$, we have an explicit form for $T$. It is the inverse map of the Joukowski function:
\[ G(z)=\frac{1}2 (z+\frac1z) \text{ (see \cite{lac_euler} for more details about this function).}\]
In this case, we can easily compute that $\OM_{\e,\y}$ is the interior of an ellipse parametrized by
\[ x(\th)=\frac{\e}2 \bigl( (1+\y)+\frac{1}{1+\y}\bigl) \cos \th,\ y(\th)=\frac{\e}2 \bigl( (1+\y)-\frac{1}{1+\y}\bigl) \cos \th.\]
Then, for small $\y$, the length of the ellipse is approximately (Taylor expansion of order 2) $\e(2+\y^2)$ whereas the higher is $\e(2\y)$. We can also see that for $\y=0$, we obtain the segment $\G_\e=[(-\e,0);(\e,0)]$.
\end{example}

We note that, for $\e$ fixed, the family defined in \eqref{T_eps_eta} verifies Assumption \ref{assump}. Therefore, we can apply directly the result obtain in \cite{lac_euler}. Let $\om_0$ be a smooth initial vorticity, compactly supported outside the obstacle. Let $\g$ be a real. The motion of an incompressible ideal flow in $\Pi_{\e,\y}$ is governed by the Euler equations:
\begin{equation*}
\left\lbrace\begin{aligned}
&\pd_t u^{\e,\y}+u^{\e,\y}\cdot \na u^{\e,\y}=-\na p^{\e,\y} &\text{ in }(0,\infty) \times {\Pi_{\e,\y}} \\
&\diver u^{\e,\y} =0 &\text{ in } [0,\infty) \times {\Pi_{\e,\y}} \\
&u^{\e,\y}\cdot \hat{n} =0 &\text{ on } [0,\infty) \times {\pd \OM_{\e,\y}}\\
&\lim_{|x|\to\infty}|u^{\e,\y}|=0 & \text{ for }t\in[0,\infty)\\
&u^{\e,\y}(0,x)=u_0^{\e,\y}(x) &\text{ in } \Pi_{\e,\y}
\end{aligned}\right.
\end{equation*}
where $p^{\e,\y}=p^{\e,\y}(t,x)$ is the pressure. In fact, to study the two dimensional ideal flows, it is more convenient to work on the vorticity equations ${\om}^{\e,\y}:= \curl u^{\e,\y}(=\pd_1 u^{\e,\y}_2-\pd_2 u^{\e,\y}_1)$ which are equivalent to the previous system: 
\begin{equation}\label{euler}
 \left\lbrace\begin{aligned}
&\pd_t {\om}^{\e,\y}+u^{\e,\y}\cdot \na {\om}^{\e,\y}=0 &\text{ in } (0,\infty) \times {\Pi_{\e,\y}} \\
&\diver u^{\e,\y} = 0, \ \curl u^{\e,\y} = \om^{\e,\y} &\text{ in } (0,\infty) \times {\Pi_{\e,\y}} \\
&u^{\e,\y}\cdot \hat{n} =0 &\text{ on } [0,\infty) \times {\pd \OM_{\e,\y}}\\
&\lim_{|x|\to\infty}|u^{\e,\y}|=0 & \text{ for }t\in[0,\infty)\\
& \oint_{\pd \OM_{\e,\y}} u^{\e,\y}\cdot \dss = \g & \text{ for }t\in[0,\infty)\\
&{\om}^{\e,\y}(0,x)={\om}_0(x) &\text{ in } \Pi_{\e,\y}.
\end{aligned}\right.
\end{equation}
The interest of such a formulation is that we recognize a transport equation. The transport nature allows us to conclude that the $L^p(\Pi_{\e,\y})$ norms of the vorticity are conserved, for $p\in [1,\infty]$, which gives us directly an estimate and a weak convergence in $L^\infty(L^p)$ for the vorticity.

For all $\e,\y$, $\OM_{\e,\y}$ is smooth and Kikuchi in \cite{kiku} states that there exists a unique pair $(u^{\e,\y},\om^{\e,\y})$ which is a global strong solution of Euler equation in $\R^+\times \Pi_{\e,\y}$ verifying
\[\om^{\e,\y}(0,\cdot)=\om_0,\  \int_{\pd \OM_{\e,\y}} u^{\e,\y}(0,s)\cdot \tau \, ds =\g \text{ and } \lim_{|x|\to \infty} u^{\e,\y}(t,x)=0 \ \forall t.\]
A characteristic of this solution is the conservation of the velocity circulation on the boundary, and that $m:= \int \curl u^{\e,\y}= \int \om_0$.

Next, we apply the result of \cite{lac_euler} to define a pair $(u^\e,\om^\e)$:
\begin{theorem}\label{lac_euler} If Assumption \ref{assump} is verified, then there exists a subsequence $\y_k\to 0$ with
\[ \F^{\e,\y} \om^{\e,\y} \rightharpoonup \om^\e \text{ weak-$*$ in } L^\infty(\R;L^4(\R^2)) \text{ and }  \F^{\e,\y} u^{\e,\y} \to u^\e  \text{ strongly in } L^2_{\loc}(\R\times \R^2)\]
where the following properties are verified
\begin{itemize}
\item $u^\e$ can be expressed in terms of $\g$, and $\om^\e$:
\[ u^\e(x)=\dfrac{1}{2\pi} DT_\e^t(x) \int_{\R^2} \Bigl(\dfrac{(T_\e(x)-T_\e(y))^\perp}{|T_\e(x)-T_\e(y)|^2}-\dfrac{(T_\e(x)- T_\e(y)^*)^\perp}{|T_\e(x)- T_\e(y)^*|^2}\Bigl)\om^\e(t,y)dy + \a\dfrac{1}{2\pi} DT_\e^t(x) \dfrac{T_\e(x)^\perp}{|T_\e(x)|^2} \]
with $\a = \int \om_0 + \g$;
\item $u^\e$ and $\om^\e$ are weak solutions of 
\begin{equation*}
\pd_t \om^\e+u^\e.\na\om^\e =0 \text{ in }\R^2\times(0,\infty).
\end{equation*}
\end{itemize}
\end{theorem}
In this theorem, $\F^{\e,\y}$ denotes a cutoff function of an $\y$-neighborhood of $\G_\e$, and we write $x^*=\frac{x}{|x|^2}$. Moreover, we can find in \cite{lac_euler} the following properties.

\begin{proposition}\label{remark_lac_euler}
Let $u^\e$ be given as in Theorem \ref{lac_euler}. For fixed $t$, the velocity
\begin{itemize}
\item[i)] is continuous on $\R^2\setminus \G_\e$ and tends to zero at infinity.
\item[ii)] is continuous up to $\G_\e\setminus \{(-\e,0);(\e,0)\}$, with different values on each side of $\G_\e$.
\item[iii)] blows up at the endpoints of the curve like $C/\sqrt{|x-(\e,0)||x+(\e,0)|}$, which belongs to $L^p_{\loc}$ for $p<4$.
\item[iv)] is tangent to the curve, with circulation $\g$.
\item[v)]  we have $\diver u^\e=0$ and $\curl u^\e=\om^\e+g^\e_{\om^\e}\d_{\G_\e}$ in the sense of distributions of $\R^2$.
\end{itemize}
The function $g^\e_{\om^\e}$ is continuous on $\G_\e$ and blows up at the endpoints of the curve $\G_\e$ as the inverse of the square root of the distance. One can also characterize $g^\e_{\om^\e}$ as the jump of the tangential velocity across $\G_\e$.
\end{proposition}

There is a sharp contrast between the behavior of ideal flows around a
small and a thin obstacle. In \cite{ift_lop_euler}, the additional term due to the vanishing obstacle appears as a
time-independent additional convection centered at $P$, whereas in the case of a thin obstacle, the correction term depends on the time. Moreover, if initially $\g =0$ we note that there is no singular term in the case of the small obstacle, whereas $g_{\om}$ again appears in the case of a thin obstacle. Physically, it can be interpreted by the fact that a wall blocks the fluids, a point not.

Therefore, for $\G_\e$ given, we have constructed a weak solution  $(u^\e,\om^\e)$ for the Euler equation outside the curve $\G_\e$ in the sense of Definition \ref{sol-curve}. Actually, it misses one property which can be establish as follow. For all $\y$, $\om^{\e,\y}$ verifies the transport equation in a strong sense, then we have the classical estimates $\| \om^{\e,\y}(t,\cdot)\|_{L^p} = \| \om_0\|_{L^p}$ for all $p\in [1,\infty]$. Then, the weak limit allows us to state that:
\[ \|\om^{\e}(t,\cdot)\|_{L^p(\Pi_{\e})}\leq \limsup_{\y\to 0} \|\om^{\e,\y}(t,\cdot)\|_{L^p(\Pi_{\e})}=\|\om_0\|_{L^p(\R^2)},\]
which means that $\om^\e \in L^{\infty}(L^1\cap L^\infty)$.

By thicken the curve, we prove here that there exists at least one weak solution of the Euler equations outside the curve. The goal of this paper is to study the limit of  $(u^\e,\om^\e)$, where  $(u^\e,\om^\e)$ is one of these solutions. An important future work will be the uniqueness of $(u^\e,\om^\e)$. Indeed, until now, it is possible that this pair depends on the way of the obstacles shrink to the curve, and the convergence  $(u^{\e,\y},\om^{\e,\y})\to (u^\e,\om^\e)$ holds true by extracting a subsequence.

\begin{remark} We have:
\[T_{\e,\y}=T_{1,\y}(x/\e).\]
This equality means that, for $\y >0$ fixed, we are exactly in the case of \cite{ift_lop_euler} (see Theorem \ref{ift_main}). Then, we can apply directly their result to extract a subsequence $\e_k \to 0$ such that 
\[ \F^{\e,\y} \om^{\e,\y} \rightharpoonup \om^\y \text{ weak-$*$ in } L^\infty(\R\times \R^2) \text{ and }  \F^{\e,\y} u^{\e,\y} \to u^\y  \text{ strongly in } L^2_{\loc}(\R\times \R^2)\]
where the following properties are verified
\begin{itemize}
\item $u^\y$ and $\om^\y$ are weak solutions of $\pd_t \om^\y+u^\y.\na\om^\y =0$ in $\R^2\times(0,\infty)$;
\item $\diver u^\y=0$ and $\curl u^\y=\om^\y+\g \d_{0}$ in the sense of distributions of $\R^2$;
\item $|u^\y|\to 0$ at infinity.
\end{itemize}
Moreover, \cite{lac_miot} establishes the uniqueness of the previous problem, then we can state that the convergence holds true without extraction of a subsequence, and the limit does not depend on $\y$:
\[ (u^{\e,\y},\om^{\e,\y})\to (u,\om) \text{ as } \e\to 0.\]
\end{remark}

\subsubsection{Comment on the previous remark}
We could take advantage of the previous remark thinking along this line: assuming that we can prove that
\begin{equation}\label{conv-unif}
 (u^{\e,\y},\om^{\e,\y})\to (u^\e,\om^\e) \text{ as } \y\to 0 \text{, uniformly in }\e,
 \end{equation}
then, for all $\rho >0$, there exists a $\y_\rho$, such that
\[ \| (u^\e,\om^\e) - (u^{\e,\y},\om^{\e,\y}) \| \leq \rho/2,\ \forall \e.\]
For this $\y_\rho$ fixed, we apply directly the result of \cite{ift_lop_euler} to find $\e_\rho$, such that for all $\e\leq \e_\rho$, we have  
\[ \| (u^{\e,\y_\rho},\om^{\e,\y_\rho}) - (u,\om) \| \leq \rho/2.\]
The fact that we find the same $(u,\om)$ for any $\y$ comes from \cite{lac_miot}. Therefore, we have found $\e_\rho$, such that for all $\e\leq \e_\rho$, we have
\[ \| (u^\e,\om^\e) - (u,\om) \| \leq  \| (u^\e,\om^\e) - (u^{\e,\y_\rho},\om^{\e,\y_\rho}) \| +  \| (u^{\e,\y_\rho},\om^{\e,\y_\rho}) - (u,\om) \| \leq \rho,\]
which is the desired result. To make this proof rigourous, we have to prove \eqref{conv-unif}, i.e. to rewrite the article \cite{lac_euler} with adding the parameter $\e$, and to check carefully that the convergence is uniform in $\e$. It is possible to check that all the estimates are uniform in $\e$, but it is difficult to give a sense to uniform convergence for the weak-$*$ topology. We would have had to establish a uniform version of Banach-Alaoglu's Theorem. Indeed, we always say: as $\| \om^{\e,\y} \|_{L^p} = \| \om_0 \|_{L^p}$, we can extract a subsequence such that $ \om^{\e,\y} \rightharpoonup \om^\e$ weak-$*$ in $L^\infty(L^p)$, as $\y\to 0$. What does it mean and how can we prove that we can choose a subsequence such that this convergence weak-$*$ is uniform in $\e$ ?

Actually, the general idea is to pass to the limit $(u^\e,\om^\e)\to(u,\om)$ with similar arguments to \cite{ift_lop_euler}. This last article cannot be used directly and, as we will see, it is not obvious to adapt their arguments. An other possibility is to adapt the arguments of \cite{lac_euler}. However, in the present case of the small curve, some estimates becomes better than in the case of the thin obstacle, and we will see that we can apply Aubin-Lions and Div-Curl Lemmas, in order to pass to the limit. This idea appeared in \cite{ift_lop_euler}, cannot be applied in \cite{lac_euler}, and we choose here to use it in our case because it goes faster than using the arguments from \cite{lac_euler}.

\subsubsection{Biholomorphism estimate}
As it was focused in the introduction, we need to estimate the biholomorphism to estimate the velocity, thanks to the Biot-Savart law.
\begin{proposition}\label{biholo-est}
The biholomorphism family   $\{T_{\e}\}$, defined in \eqref{T_eps}, verifies
\begin{itemize}
\item[(i)] $\e^{-2} \det(DT_{\e}^{-1})$ is bounded on $D^c$ independently of $\e$,
\item[(ii)] for any $R>0$, for all $p\in [1,4)$, $\e \|DT_\e\|_{L^p(B(0,R))}$ is bounded uniformly in $\e$,
\item[(iii)] for $R>0$ large enough, there exists  $C_R>0$ such that $|DT_{\e}(x)|\leq \dfrac{C_R}\e$ on $B(0,R)^c$.
\end{itemize}
\end{proposition}

\begin{proof} The point (iii) is straightforward using \eqref{T_eps} and Proposition  \ref{2.2}.

Concerning (i), we directly see that
\[T_{\e}^{-1}(x) = \e T^{-1}(x),\]
hence
\[ \det(DT_{\e}^{-1}) (x) = \e^2 \det(DT^{-1}) (x),\]
which prove point (i), thanks to Proposition \ref{2.2}.

For the last point, we compute
\begin{eqnarray*}
 \Bigl( \int_{B(0,R)} \bigl| \frac1\e DT(\frac{x}\e) \bigl|^p\, dx\Bigl)^{1/p}= \e^{\frac2p - 1} \Bigl( \int_{B(0,R/\e)} \bigl|  DT(y) \bigl|^p\, dy\Bigl)^{1/p}.
\end{eqnarray*}
Using Proposition \ref{2.2}, we know that $DT$ belongs in $L^p (0,R_1)$ for all $R_1>0$ and all $p<4$. However, we should also take care of the behavior of $DT$ at infinity, because $\lim_{\e\to 0} R/\e =\infty$. Remark \ref{2.5} allows us to pretend that there exists $\tilde R$ such that
\[ |DT(y)| \leq \b+1 \text{ for } |y|\geq \tilde R.\] 
Therefore, for $p<4$
\begin{eqnarray*}
\Bigl( \int_{B(0,R/\e)} \bigl|  DT(y) \bigl|^p\, dy\Bigl)^{1/p}&\leq&\Bigl( \int_{B(0,\tilde R)} \bigl|  DT(y) \bigl|^p\, dy\Bigl)^{1/p}+\Bigl( \int_{B(0,R/\e)\setminus B(0,\tilde R)} ( \b+1 )^p\, dy\Bigl)^{1/p} \\
&\leq& C_p + \frac{C_\b R^{2/p}}{\e^{2/p}},
\end{eqnarray*}
which means that, for $\e$ small enough, independently of $R$,
\begin{equation}\label{T_Lp}
\|DT_\e\|_{L^p(B(0,R)\cap\Pi_{\e,\y})}  \leq \e^{- 1} C(1 + R^{2/p}),
\end{equation}
which ends the proof.
\end{proof}

\subsection{Biot-Savart law}\label{sect:biot}\ 

One of the key of the study for two dimensional ideal flow is to work with the vorticity equation, which is a transport equation. For example, in the case of a smooth obstacle, we choose initially $\om_0 \in L^1\cap L^\infty$, then $\| \om^\e(t,\cdot)\|_{L^p}= \| \om_0\|_{L^p}$ for all $t$. Next, Banach-Alaoglu's Theorem allows us to extract a subsequence such that $\om^\e \rightharpoonup \om$ weak-$*$. So, we have some estimates and weak-$*$ convergence for the vorticity, and the goal is to establish estimates and strong convergence for the velocity. For that, we introduce the Biot-Savart law, which gives the velocity in terms of the vorticity. Another advantage of the two dimensional space is that we have explicit formula, thanks to complex analysis and the identification of $\R^2$ and $\C$.

Let $\OM$ be a bounded, connected, simply connected subset of the plane. We denote by $\Pi$ the exterior domain: $\Pi:= \R^2\setminus\overline{\OM}$, and let $T$ be a biholomorphism between $\Pi$ and $(\overline D)^c$ such that $T(\infty)=\infty$.

We denote by $G_{\Pi}=G_{\Pi} (x,y)$ the Green's function, whose the value is:
\begin{equation}
\label{green}
G_{\Pi}(x,y)=\frac{1}{2\pi}\ln \frac{|T(x)-T(y)|}{|T(x)-T(y)^*||T(y)|}
\end{equation}
writing $x^*=\frac{x}{|x|^2}$. The Green's function verifies: 
\begin{equation*}
\left\lbrace \begin{aligned}
&\D_y G_{\Pi}(x,y)=\d(y-x) \text{ for }x,y\in \Pi \\
&G_{\Pi}(x,y)=0 \text{ for }y\in\G \\
&G_{\Pi}(x,y)=G_{\Pi}(y,x)
\end{aligned}
\right.
\end{equation*}
\begin{remark}\label{T-unique}
In fact, the Green's function is unique, even in the case where $\OM$ is a curve $\G$. Indeed, let $F_1$ and $F_2$ be two biholomorphisms from $\Pi$ to the exterior of the unit disk. Then $F_1\circ F_2^{-1}$ maps $D^c$ to $D^c$ and we can apply the uniqueness result for Riemann mappings to conclude that there exists $a\in \C$ such that $|a|=1$ and $F_1=aF_2$. Moreover, changing $T$ by $aT$ in \eqref{green} does not change the Green's function. The uniqueness of the Green's  function outside the unit disk gives us the result.
\end{remark}

The kernel of the Biot-Savart law is $K_{\Pi}=K_{\Pi}(x,y) := \na_x^\perp G_{\Pi}(x,y)$. With $(x_1,x_2)^\perp=\begin{pmatrix} -x_2 \\ x_1\end{pmatrix}$, the explicit formula of $K_{\Pi}$ is given by 
\begin{equation*}
K_{\Pi}(x,y)=\dfrac{1}{2\pi} DT^t(x)\Bigl(\dfrac{(T(x)-T(y))^\perp}{|T(x)-T(y)|^2}-\dfrac{(T(x)- T(y)^*)^\perp}{|T(x)- T(y)^*|^2}\Bigl).
\end{equation*}
We require information on far-field behavior of $K_{\Pi}$. We will use several times the following general relation:
\begin{equation}
\label{frac}
\Bigl| \frac{a}{|a|^2}-\frac{b}{|b|^2}\Bigl|=\frac{|a-b|}{|a||b|},
\end{equation}
which can be easily checked by squaring both sides.

We now find the following inequality:
\[ |K_{\Pi}(x,y)|\leq C \frac{|T(y)-T(y)^*|}{|T(x)-T(y)||T(x)-T(y)^*|}.\]

For $f\in C_c^\infty({\Pi})$, we introduce the notation
\[K_{\Pi}[f]=K_{\Pi}[f](x):=\int_{\Pi} K_{\Pi}(x,y)f(y)dy,\]
and we have for large $|x|$ that 
\begin{equation}\label{K-inf}
 |K_{\Pi}[f]|(x)\leq \dfrac{C_1}{|x|^2},
 \end{equation}
 where $C_1$ depends on the size of the support of $f$. We have used here Remark \ref{2.5}.

The vector field $u=K_{\Pi}[f]$ is a solution of the elliptic system:
\begin{equation*}
\left\lbrace\begin{aligned}
\diver u &=0 &\text{ in } {\Pi} \\
\curl u &=f &\text{ in } {\Pi} \\
u\cdot \hat{n}&=0 &\text{ on } \pd \OM \\
\lim_{|x|\to\infty}|u|&=0
\end{aligned}\right.
\end{equation*}

Let $\hat{n}$ be the unit normal exterior to $\Pi$. In what follows all contour integrals are taken in the counter-clockwise sense, so that $\int_{\pd \OM} F\cdot \dss=-\int_{\pd \OM} F\cdot \hat{n}^\perp ds$. Then the harmonic vector fields 
\[H_{\Pi} (x)=\frac{1}{2\pi}\na^\perp \ln |T(x)|= \frac{1}{2\pi} DT^t(x)\frac{T(x)^\perp}{|T(x)|^2}\]
is the unique\footnote{see e.g. \cite{ift_lop_euler}.} vector field verifying
\begin{equation*}
\left\lbrace\begin{aligned}
\diver H_{\Pi} &=0 &\text{ in } {\Pi} \\
\curl H_{\Pi} &=0 &\text{ in } {\Pi} \\
H_{\Pi}\cdot \hat{n}&=0 &\text{ on } \pd \OM \\
H_{\Pi} (x)\to 0 &\text{ as }|x|\to\infty\\
\int_{\G} H_{\Pi} \cdot \dss &=1.
\end{aligned}\right.
\end{equation*}
It is also known that $H_{\Pi}(x)=\mathcal{O}(1/|x|)$ at infinity. Therefore, putting together the previous properties, we obtain that the vector field $u^\e$:
\begin{equation}\label{biot}
u^\e(x):= K_\e[\om^\e](x) + (\g+\int \om^\e) H_\e(x)
\end{equation}
is the unique vector fields which verifies 
\begin{equation*}
\left\lbrace\begin{aligned}
\diver u^\e &=0 &\text{ in } {\Pi_\e} \\
\curl u^\e &= \om^\e &\text{ in } {\Pi_\e} \\
u^\e \cdot \hat{n}&=0 &\text{ on } \G_\e \\
u^\e(x)\to 0 &\text{ as }|x|\to\infty\\
\int_{\G_\e} u^\e \cdot \dss &= \g.
\end{aligned}\right.
\end{equation*}
Concerning the uniqueness, we note that Remark \ref{T-unique} allows us to apply the theory developed in \cite{ift_lop_euler} (see Lemma 3.1 therein, which is a consequence of Kelvin's Circulation Theorem). 

\begin{remark}\label{rem : prop-curve}
This property means that the velocity of any weak solution (in the sense of Definition \ref{sol-curve}) can be written as \eqref{biot}. Adding the fact that $\om^\e$ belongs to $L^\infty(L^1\cap L^\infty)$, Proposition \ref{remark_lac_euler} states that \eqref{biot} and the behavior of $T_\e$ (see Proposition \ref{2.2}) imply some interesting properties on the velocity. For example, we infer that $u^\e$ has a jump across $\G_\e$ and blows up near the end-points of the curve as the inverse of the square root of the distance.
\end{remark}

Before working on the convergence $(u^\e,\om^\e)\to (u,\om)$, let us prove that Assumption \ref{assump} is verified for good geometrical convergence $\OM_\y \to \G$, which will complete \cite{lac_euler,lac_NS}.

\subsection{Proof of Assumption \ref{assump}.}\label{assump-proof}\ 

For this subsection, we consider, as Pommerenke in \cite{pomm-1,pomm-2}, that a domain is an open, connected subset of $\C$. We identify also $\C$ and $\R^2$. 

Let $\G$ be an open Jordan arc which verifies the assumptions of Proposition \ref{2.2}. Let $\{\OM_n\}$ be a sequence of bounded, open, connected, simply connected subset of the plane such that $\G \subset \OM_n$ and $\pd \OM_n$ is a smooth Jordan curve. We denote his complementary by $\Pi_n:= \R^2\setminus \overline{\OM_n}$. We recall that $D:= B(0,1)$, and we set $\D:= \R^2 \setminus \overline{D}$. For all $n$, we denote by $T_n$ the unique conformal mapping from $\Pi_n$ to $\D$ which sends $\infty$ to $\infty$, $\pd \OM_n$ to $\pd D$ and such that $T_n'(\infty)\in \R^+_*$ (Riemann mapping theorem). By Remark \ref{T-unique}, we consider the unique  $T$ which maps $\Pi:= \G^c$ to $\D$, such that $T(\infty)=\infty$ and $T'(\infty)\in \R^+_*$. The properties of $T$ are listed in Proposition \ref{2.2}. In order to apply the theory of domain convergence to smooth domains, we work here in the image of $T$, where there is no end-point. By continuity of $T$, we state that $\tilde \Pi_n:= T(\Pi_n)$ is an exterior domain, which means that $\tilde \OM_n := \R^2\setminus \overline{\tilde \Pi_n}$ is a smooth, bounded, open, connected, simply connected subset of the plane, containing $\overline{D}$.

\begin{figure}[ht]
\includegraphics[height=5.5cm]{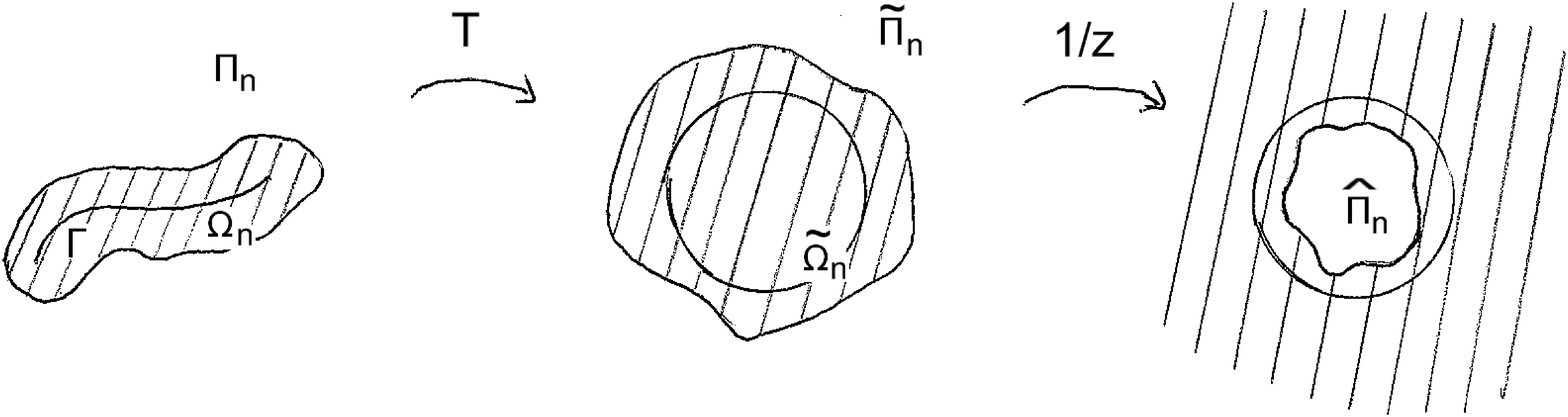}
 \end{figure}

\vspace{-1.5cm}
\centerline{{\bf PICTURE 1}: image of $T$.}

\vspace{12pt}

By the Riemann mapping theorem, we have that $g_n:= T_n\circ T^{-1}$ is the unique univalent function (meromorphic and injective) mapping $\tilde \Pi_n$ to $\D$ and satisfying $g_n(\infty)=\infty$, arg $g_n'(\infty)=0$. As $g_n(\infty)=\infty$, it means that $g_n$ is analytic in $\D$. To apply directly some results from \cite{pomm-1}, we also introduce Riemann mappings in bounded domains:
\[f_n(z):= \frac{1}{g_n(1/z)},\]
which maps $\hat \Pi_n:= 1/T(\Pi_n)$ to $D$, verifying $f_n(0)=0$ and arg $f_n'(0)=0$ (see Picture 1).

A convergence of $\OM_n$ to $\G$ can be translated by a convergence of $\tilde \OM_n$ to $D$. The goal of this subsection is to define the geometric convergence in order that the properties cited in Assumption \ref{assump} are verified. In other word, we have already an example of an obstacle family where they are verified:  
\[ \tilde \OM_n:= B(0,1+\frac{1}n),\]
(see \eqref{T_eps_eta}), and the issue here is to find more families where the results of \cite{lac_euler,lac_NS} hold.

Concerning univalent functions, the first convergence of domain was introduced by Carath\'eodory in 1912.
\begin{definition}
Let $w_0\in \C$ be given and let $G_n$ be domains with $w_0\in G_n\subset \C$. We say that
$G_n\to G$ as $n\to \infty$ with respect to $w_0$, in the sense of {\it kernel convergence} if
\begin{itemize}
\item[(i)] either $G=\{ w_0\}$, or $G$ is a domain $\neq \C$ with $w_0\in G$ such that some neighborhoods of every $w\in G$ lie in $G_n$ for large $n$;
\item[(ii)] for $w\in \pd G$ there exist $w_n\in \pd G_n$ such that $w_n\to w$ as $n\to \infty$.
\end{itemize}
\end{definition}
It is clear that the limit is unique. This notion is very general, and it follows the Carath\'eodory kernel theorem: 
\begin{theorem} Let the functions $h_n$ be analytic and univalent in $D$, and let $h_n(0)=0$, $h'_n(0)>0$, $H_n = h_n(D)$. Then $\{ h_n \}$ converges locally uniformly in $D$ if and only if $\{H_n\}$ converges to its kernel $H$ and if $H\neq\C$. Furthermore, the limit function maps $D$ onto $H$.
\end{theorem}

We remark that this theorem does not concern uniform convergence of biholomorphism up to the boundary. Indeed, we need such a property for points (i), (ii), (iii) in Assumption \ref{assump}. Moreover, convergence in the kernel sense allows some strange examples of families, as a fold on the boundary (see Picture 2). It seems unbelievable that $\{T_n\}$ verifies Assumption \ref{assump} near the boundary.

\begin{figure}[ht]
\includegraphics[height=4cm]{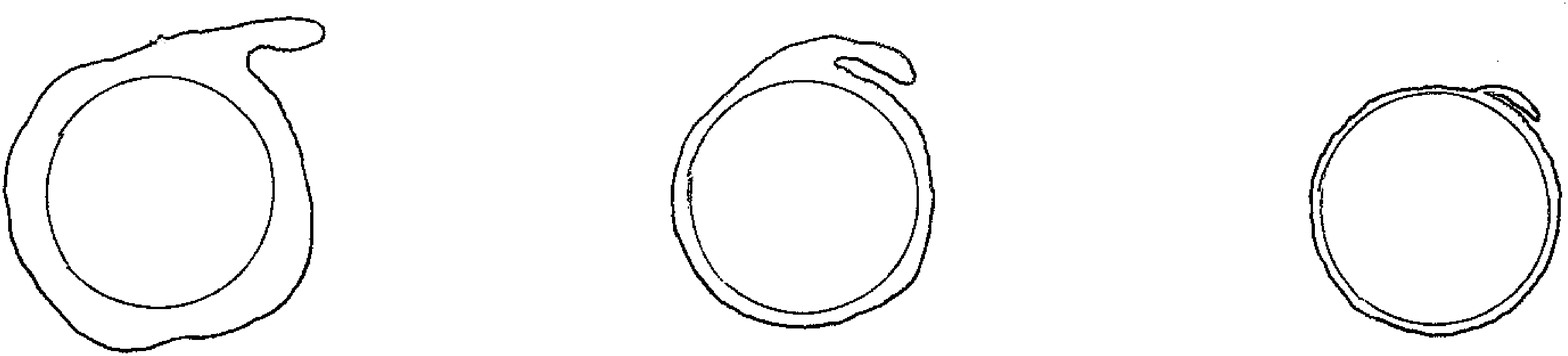}
 \end{figure}

\vspace{-1cm}
\centerline{{\bf PICTURE 2}: example of a family which converges in the kernel sense.}

\vspace{12pt}

To prevent such a case, we add another definition.

\begin{definition} A sequence $\{A_n\}$ of compact sets in $\C$ is called {\it uniformly locally connected} iff for every $\e>0$, there exists $\d >0$ (independent of $n$) such that, if $a_n,b_n\in A_n$ and $|a_n-b_n|<\d$, then we can find  connected compact sets $B_n$ with
\[a_n,b_n\in B_n \subset A_n, \ \textrm{diam } B_n < \e, \ \forall n.\]
\end{definition}

This kind of definition does not allow the case of Picture 2. However, the examples of Pictures 3 and 4 are authorized, which are classical shapes in rugosity theory.

\begin{figure}[ht]

                \begin{minipage}[t]{6cm}
\hspace{1cm}
\includegraphics[height=5cm]{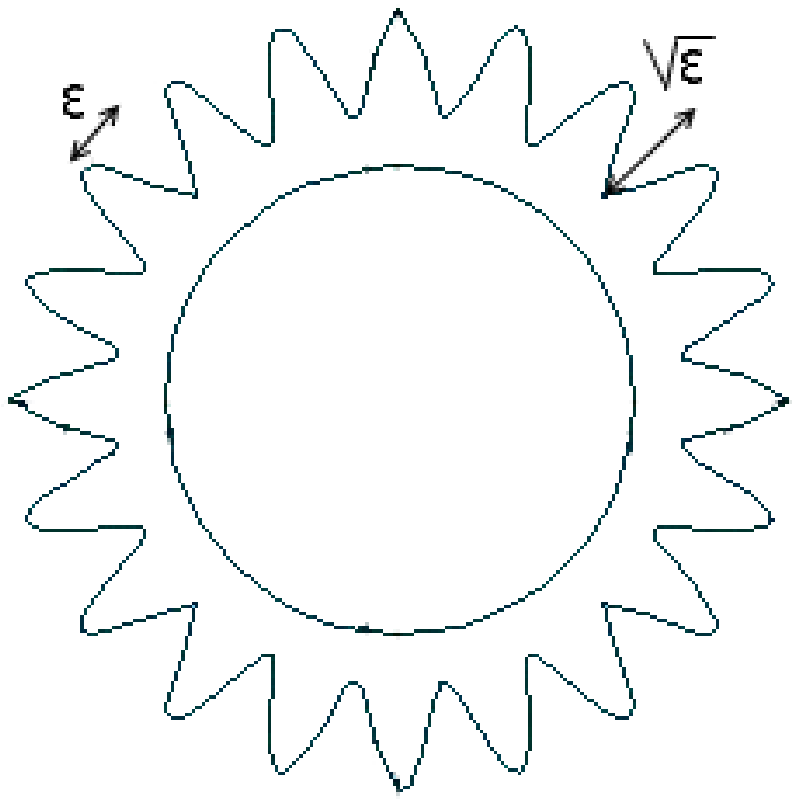}

\centerline{\ \hspace{1.2cm}{\bf PICTURE 3}}   
                \end{minipage}
                \hfill
                \begin{minipage}[t]{6cm}
\hspace{-1cm}
\includegraphics[height=5cm]{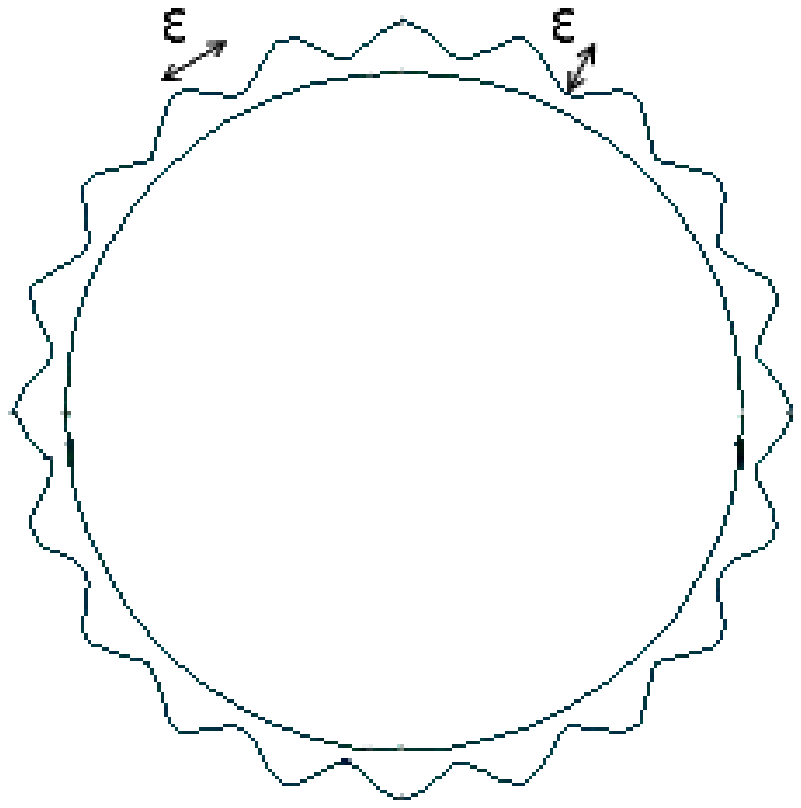}

\centerline{{\bf PICTURE 4}\hspace{2.5cm}\ }
                \end{minipage}
                \end{figure}

Thanks to this definition, we can cited Theorem 9.11 of \cite{pomm-1}, which extends to the uniform convergence up to the boundary.
\begin{theorem}\label{theo 9.11}
Let the functions $g_n$ and $g$ be univalent in $\D$ and continuous in $\overline{\D}$, and let $g_n(\infty)=\infty$ and $E_n=\C\setminus g_n(\D)$. Suppose that $g_n\to g$ locally uniformly in $\D$. Then this convergence is uniform in $\overline{\D}$ if and only if the sequence $(E_n)$ is uniformly locally connected.
\end{theorem}

We apply these theorems to obtain a part of the properties of Assumption \ref{assump}.

\begin{proposition} \label{prop-1}
Let us assume that there exists $R_0>1$ such that $\tilde \OM_n \subset B(0,R_0)$ for all $n$. If $\tilde \Pi_n \to \D$ in the sense of kernel convergence with respect to $\infty$, such that $\{ \tilde \OM_n \}$ is uniformly locally connected, then there exists $R_1>0$ such that
\begin{itemize}
\item[(a)] $\|(g_n(z)- z)/|z|\|_{L^\infty(\tilde \Pi_n)}\to 0$ as $n \to \infty$,
\item[(b)] $(g_n^{-1})'$ is bounded on $B(0,R_1)^c$ independently of $n$,
\item[(c)] $\|g_n'(z) - 1\|_{L^\infty(B(0,R_1)^c)}\to 0$ as $n \to \infty$,
\item[(d)] there exists  $C>0$ such that $|g_n'' (z)|\leq \frac{C}{|z|}$ on $B(0,R_1)^c$.
\end{itemize}
\end{proposition}

\begin{proof}
After noting that $\hat \Pi_n \to D$ in the kernel sense with respect to $0$, we apply the Carath\'eodory kernel theorem to $f_n^{-1}$. Riemann mappings are biholomorphisms, then they are univalents and analytics. We deduce that $f_n^{-1}$ converges uniformly to Id  in each compact subset of $D$. As $g_n^{-1}(z) =1/f_n^{-1}(1/z)$, we also know that $g_n^{-1}$ converges locally uniformly in $\D$. Moreover, Kellogg-Warschawski Theorem (see Theorem 3.6 of \cite{pomm-2}) allows us to state that $g_n^{-1}$ is continuous in $\overline{\D}$ because the boundary $\pd \tilde \OM_n$ is smooth. Therefore, we use Theorem \ref{theo 9.11} to state that $g_n^{-1}$ converges uniformly to Id in $\overline{\D}$. In particularly, $g_n^{-1}$ converges uniformly to Id in $B(0,2)\setminus D$, which means that $f_n^{-1}$ tends to Id uniformly in $\overline{D}\setminus B(0,1/2)$. Adding the fact that $f_n^{-1}\to$Id uniformly in $\overline{B(0,3/4)}$, we obtain that this convergence is uniformly in $\overline{D}$.

A consequence of this uniform convergence, is that $f_n\to$Id uniformly. Indeed, for all $z\in \hat \Pi_n$, we have
\[ |f_n(z) - z|= |y_n - f_n^{-1}(y_n)|\leq \| \textrm{Id} - f_n^{-1}\|_{L^\infty (\overline{D})},\]
which means that $\|f_n - \textrm{Id} \|_{L^\infty (\hat \Pi_n)}\to 0$ as $n\to \infty$. Let us prove that it follows Point (a). Near the boundary, we easily compute that
\[ \Bigl\| \frac{|g_n(z)- z|}{|z|}\Bigl\|_{L^\infty(\tilde \Pi_n\cap B(0,3R_0))} =  \Bigl\| \frac{|\frac1z- f_n(\frac1z)|}{|f_n(\frac1z)|}\Bigl\|_{L^\infty(\tilde \Pi_n\cap B(0,3R_0))} =  \Bigl\| \frac{|y-f_n(y)|}{|f_n(y) |}\Bigl\|_{L^\infty(\hat \Pi_n \setminus B(0,\frac{1}{3 R_0}))}.\]
As, $\|f_n - \textrm{Id} \|_{L^\infty (\hat \Pi_n)}\to 0$ as $n\to \infty$, it means that there exists $N_1$ such that for all $n>N_1$, we have $|f_n(z)|> \frac{1}{4R_0}$ for all $z\in \hat \Pi_n\setminus B(0,\frac{1}{3R_0})$. Therefore, for all $n>N_1$, we obtain that 
\begin{equation}\label{g-bd}
 \| |g_n(z)- z|/|z|\|_{L^\infty(\tilde \Pi_n\cap B(0,3R_0))} \leq 4R_0 \| \textrm{Id} -f_n \|_{L^\infty(\hat \Pi_n)}.
 \end{equation}
Far the boundary, first we prove that $f_n(z)/z$ converge uniformly to $1$ in $B(0,1/(2R_0))$. As $f_n(0)=0$, we know that the map $z\mapsto f_n(z)/z$ is holomorphic in $D$. After remarking that $1/(2R_0)<1/R_0<1$, then we can write the Cauchy formula to state that:
\begin{eqnarray*}
\forall z\in B(0,1/(2R_0)), \  \frac{f_n(z)}{z} -1 &=& \frac{1}{2\pi i} \oint_{\pd B(0,1/R_0)} \frac{f_n(y)/y-1}{y-z}\, dy\\
\Bigl| \frac{f_n(z)}{z}-1\Bigl| &\leq & 2 R_0^2 \| f_n - \textrm{Id} \|_{L^\infty(\overline{B(0,1/R_0)})},
\end{eqnarray*}
which means that $f_n(z)/z$ converges uniformly to $1$ in $B(0,1/(2R_0))$. Finally, we write that 
\begin{equation}\label{g-inf}
 \| |g_n(z)- z|/|z|\|_{L^\infty( B(0,2R_0)^c)} = \| \frac{y}{f_n(y)} -1 \|_{L^\infty( B(0,\frac{1}{2R_0}))}.
 \end{equation}
Putting together \eqref{g-bd} and \eqref{g-inf}, we end the proof of point (a).

Now, we focus on uniform convergence of derivatives. As $f_n^{-1}$ is holomorphic in $D$, we can write the Cauchy formula for all $z\in B(0,1/2)$
\begin{eqnarray*}
 (f_n^{-1})'(z)-1 &=& \frac{1}{2\pi i} \oint_{\pd B(0,3/4)} \frac{f_n^{-1}(y)-y}{(y-z)^2}\, dy\\
|( f_n^{-1})'(z)-1| &\leq & 16 \| f_n^{-1}-\textrm{Id} \|_{L^\infty(B(0,3/4))}.
\end{eqnarray*}
which gives that $(f_n^{-1})'$ converges also uniformly to $1$ in $B(0,1/2)$. Differentiating $g_n^{-1}(z) = 1/f_n^{-1}(1/z)$, we obtain $(g_n^{-1})'(z)=\frac{(f_n^{-1})'(1/z)}{(zf_n^{-1}(1/z))^2}$. Therefore  $(g_n^{-1})'$ converges uniformly to $1$ in $B(0,2)^c$, which implies point (b).

The other points can be established in the same way. Indeed, $f_n\to $Id uniformly in $B(0,1/R_0)$, then using again the Cauchy formula, we prove that $f_n'\to 1$ uniformly in $B(0,1/(2R_0))$. Using that $(g_n)'(z)=\frac{(f_n)'(1/z)}{(zf_n (1/z))^2}$, we pretend that $g_n'\to 1$ uniformly in $B(0,2R_0)^c$ which implies (c). Finally, we use again the Cauchy formula to get that $f_n'' \to 0$ uniformly in $B(0,1/(2R_0))$. Next, we differentiate once more the relation between $g_n$ and $f_n$, and we compute:
\[ g_n''(z)=-\frac{f_n''(1/z)+ 2 f_n'(1/z)(zf_n(1/z)-f_n'(1/z))}{z^4 (f_n(1/z))^3}.\]
Using all the uniform convergences of $f_n$, $f'_n$ and $f_n''$ in $B(0,1/(2R_0))$, we obtain (d) with $R_1=2 R_0$.

This ends the proof.
\end{proof}

Using the relation, $T_n=g_n \circ T$, it will not be complicated to obtain properties (i), (iv) and (v) of Assumption \ref{assump}. However, seeing points (ii) and (iii), we remark that it misses us uniform convergence of the derivatives up to the boundary. In the previous proof, we obtain uniform convergence of $(f_n^{-1})'$ in all the compact subsets, thanks to the Cauchy formula, finding a curve between the compacts and $\pd D$. Such a proof can not be used to obtain convergence up to the boundary. Therefore, we need to present the results of \cite{war}. In this article, it appears that we can except a uniform convergence of $(f^{-1}_n)'$ only if we assume some convergences of the tangent vectors of $\pd \hat \Pi_n$ to the tangent vector of $\pd D$.

For that, we remark that $C_n:= \pd \hat \Pi_n$ denotes a closed Jordan curve which posses continuously turning tangents. We also put $C:= \pd D$. Let $\t_n(s)$ and $\t$ be their tangent angles, expressed as functions of the arc length. We gives now the assumptions needed to present the main theorem of \cite{war}. We assume that there exists $\e>0$ such that:
\begin{enumerate}
\item $C_n$ is in the $\e$-neighborhood of $C$, i.e. any point of $C_n$ is contained in a circle of radius $\e$ about some points of $C$.
\item If $\D s$ denotes the (shorter) arc of the curve $C_n$ between $w'$ and $w''$, then
\[\frac{\D s}{|w'-w''|}\leq c.\]
\item For any point $w_n\in C_n$, pertaining to the arc length, $0\leq s_n\leq L_n$, there corresponds a point $w\in C$, pertaining to the arc length $\s=\s(s_n)$, such that $|w_n-w|\leq \e$ and that, for suitable choise of the branches,
\[ \sup_{0\leq s_n\leq L_n}  |\t_n(s_n) - \t(\s(s_n))| \leq q\e.\]
\item $\t_n$ is H\"{o}lder continuous, i.e. there exist $k>0$, $\a \in (0,1]$ such that 
\[ |\t_n(s_1)-\t_n(s_2)| \leq k(s_1-s_2)^\a, \forall s_1,s_2 \in [0,L_n].\]
\end{enumerate}
In (3), $L_n$ denotes the total length of $C_n$.

\begin{remark} In \cite{war}, the author also assumes that $C$ and $C_n$ are contained in the ring $0<d\leq |w|\leq D$. This property is directly verified in our case, with $D=1$ and $d= 1-\e$. Moreover, he assumes that $C$ is in an $\e$-neighborhood of $C_n$. Indeed, in general case, point (ii) does not imply it. However, in our case we have $D\subset \tilde \OM_n$, and we obtain such a property.
\end{remark}

\begin{remark}\label{rugo}
We see that point (2) prevents us the convergence as in Picture 3. Indeed, in the case of this picture, we have $\frac{\D s}{|w'-w''|}\approx C/\e$, which tends to $\infty$ as $\e\to 0$. In literature, a curve which verified (2) is called a chord-arc curve, and we will prove that the domains outside this kind of curve are uniformly locally connected.

Moreover, we cannot either consider the case of Picture 4, because of point (3).
\end{remark}

Then we give Theorem IV of \cite{war}:
\begin{theorem} \label{war}
If $C_n$ satisfies hypotheses (1)-(4), then
\[ \sup_{|z|\leq 1 } |(f_n^{-1})'(z)-1|\leq K \e \ln \frac{\pi}{\e}\]
where $K$ depends only on $c$, $k$, $\a$ and $q$.
\end{theorem}

Putting together Proposition \ref{prop-1} and Theorem \ref{war}, we can prove the main theorem of this subsection:

\begin{theorem}\label{theo-assump}
Let $\G$ be an open Jordan arc which verifies the assumptions of Proposition \ref{2.2}. Let $\{\OM_n\}$ be a sequence of smooth, bounded, open, connected, simply connected subset of the plane such that $\G \subset \OM_n$. For all $n$, we denote by $T_n$ the unique conformal mapping from $\Pi_n$ to $\D$ such that $T_n(\infty)=\infty$ and $T_n'(\infty)\in \R^+_*$. Let $T$ the biholomorphism constructed in Proposition \ref{2.2}, such that  $T'(\infty)\in \R^+_*$. We also denote $\hat \Pi_n:= 1/T(\Pi_n)$.

For all $n$, we assume that $C_n:= \pd \hat \Pi_n$ verifies hypotheses (1)-(4), with $c$, $k$, $\a$ and $q$ independent of $n$, and where $\e_n \to 0$ as $n\to 0$. Then, the family $\{T_n\}$ verifies Assumption \ref{assump}.
\end{theorem}

\begin{proof} We use as before the conformal mappings $g_n$ and $f_n$. Let us first note that we can apply Proposition \ref{prop-1}. Indeed, the condition (1) implies that $\hat \Pi_n \to D$ in the kernel sense with respect to $0$, which means that $\tilde \Pi_n \to \D$ in the kernel sense with respect to $\infty$. Moreover, choosing $R_0=1+\sup_n \e_n$, we also see that $\tilde \OM_n \subset B(0,R_0)$ for all $n$. Finally, we can easily check that the condition of uniformly locally connected comes from the condition (2): for every $\e$, we choose $\d =  \e/c$. Indeed, for all $a_n$ $b_n$ such that $|a_n-b_n|<\d$, (2) means that there exists $B_n\subset \tilde \OM_n$ with $a_n,b_n\in B_n$ and diam$B_n\leq c |a_n-b_n|< \e$.

Therefore, we use directly Proposition \ref{prop-1}. Theorem \ref{war} states that the convergence of $(f_n^{-1})'$ to $1$ is uniform in $\overline{D}$. Thanks to the relation between $(f_n^{-1})'$ and $(g_n^{-1})'$ (see the proof of Proposition \ref{prop-1}, it means that $(g_n^{-1})'\to 1$ uniformly in $\overline{\D}$. Adding that $g_n'(z)=1/(g_n^{-1})'(g_n(z))$, we conclude that Theorem \ref{war} allows us to extend (b) and (c) up to the boundary, i.e.
\begin{itemize}
\item[(b')] $(g_n^{-1})'$ is bounded on $\D$ independently of $n$,
\item[(c')] $\|g_n'(z) - 1\|_{L^\infty(\tilde \Pi_n)}\to 0$ as $n \to \infty$.
\end{itemize}

In order to finish this proof, we just have to compose by $T$. As $T_n=g_n\circ T$, it is obvious that
\[ \Bigl\| \frac{T_n-T}{|T|}\Bigl\|_{L^\infty(\Pi_n)} = \Bigl\| \frac{g_n-\textrm{Id}}{|\textrm{Id}|}\Bigl\|_{L^\infty(\tilde \Pi_n)}\]
which tends to zero as $n\to \infty$ by (a) of Proposition \ref{prop-1}. This proves Point (i) of Assumption \ref{assump}.

Differentiating the relation between $T_n^{-1}$ and $g_n^{-1}$, we obtain that 
\[(T_n^{-1})'(z)= (T^{-1})'(g_n^{-1}(z)) (g_n^{-1})' (z).\]
Then, it is easy to conclude that (ii) follows from (b') and Proposition \ref{2.2}. 
 
 For any $R>0$, 
 \begin{eqnarray*}
  \| T'_n-T' \|_{L^3(B(0,R)\cap \Pi_n)} &=&  \| g'_n(T(z)) T'(z) -T'(z) \|_{L^3(B(0,R)\cap \Pi_n)}\\
  &\leq& \|g'_n(z)-1\|_{L^\infty(\tilde \Pi_n)}  \| T' \|_{L^3(B(0,R)\cap \Pi_n)}
  \end{eqnarray*}
  which proves (iii), thanks to (c') and Proposition \ref{2.2}.
  
Differentiating once the relation between $T_n$ and $g_n$, and using (c) and Proposition \ref{2.2}, we obtain (iv). Differentiating once more, we have
\[ T''_n(z)=g_n''(T(z)) (T'(z))^2+g_n'(T(z)) T''(z).\]
Using (d), Proposition \ref{2.2}, (c) and Remark \ref{2.5}, we finally get (v), which ends the proof.
\end{proof}

\begin{remark} We easily see at the end of the previous proof that we can prove that:
\[\text{for any $R>0$, $p<4$, $\|DT_n - DT\|_{L^p(B(0,R)\cap \Pi_n)}\to 0$ as $n \to \infty$.}\]
\end{remark}

Theorem \ref{main 2} follows from Theorem \ref{theo-assump} and \cite{lac_euler} (see Subsection \ref{thicken} for the details).

\subsubsection*{Comment on this domain convergence}

This theorem completes \cite{lac_euler,lac_NS} in the following sense: if $C_n:= \pd T( \OM_n)$ verifies hypotheses (1)-(4), then solutions of the Euler equations (respectively Navier-Stokes) in the exterior of $\OM_n$ converge to the solution of the Euler equations (respectively Navier-Stokes) outside the curve. To make this result more attractive, it should be better to give geometric properties on $\pd \OM_n$ instead of $\pd T( \OM_n)$, but this translation is not so easy. Of course, by continuity of $T$, condition (1) can be assumed on $\pd \OM_n$, which is not the case for (2)-(4), where we give some properties of the tangent vectors. Conditions (2)-(4), prevent the pointed parts as in Picture 2. However, studying the example in \eqref{T_eps_eta}, we see that there is a pointed part near the end-point $-1$, which is mapped by $T$ to the circle $B(0,1+\e)$. In other word, $\pd \OM_n$ cannot satisfy (2)-(4) near the end-points, but $T$ can straighten it to a curve which verifies (2)-(4). This straightening up should hold for some shapes of pointed part, but surely not for any shape. We have a big family of authorized shapes: $T^{-1}(C_n)$
 where $C_n$ verifies (1)-(4). A possible result could be stated like that: 
 {\it
 \begin{itemize}
 \item if there exists $\d>0$ such that $(\pd \OM_n)\setminus (\cup B(\pm 1,\d))$ satisfy hypothesis (1)-(4) (where we replace $C$ by $\G$),
 \item if $\pd \OM_n$ corresponds to an ``authorized shape'' in $\cup B(\pm 1,\d)$
 \end{itemize}
 then Assumption \ref{assump} is verified.}
 

\subsubsection*{Comment on the convergence to smooth domains}

If we replace the convergence of domain to a curve by a convergence to a smooth domain $\OM$, we can adapt easily Theorem \ref{theo-assump} and \cite{lac_euler} in order to establish that the solution of Euler equations in (or outside) $\OM_n$ converges to the solution of Euler equations in (or outside) $\OM$, if $\OM_n\to\OM$ in the sense of hypothesis (1)-(4). Then, a consequence of this work is that we have found a sense for the domain convergences (hypothesis (1)-(4)), in order that the limit solution is a solution of Euler equations. However, we see in Remark \ref{rugo} that the classical shape in rugosity problems (Picture 4) is not allowed in our analysis. Concerning bounded domains (simply connected), Taylor in \cite{taylor} proves that we do not need so strong properties. Passing to the limit with a weaker domain convergence, he can show the existence of weak solutions for Euler equations in a non-smooth convex domain (with Lipschitz boundary). Therefore, our result does not improve the case of bounded domain, but it gives a new result in exterior domains (outside one obstacle). In exterior domain, the analysis is harder because of the behavior at infinity: the velocity is not square integrable.

\newpage

\section{{\it A priori} estimates}\label{sect : 3}

\subsection{Vorticity estimate}\label{om-est}\  

The study of two dimensional ideal flow is based on velocity estimates thanks to vorticity estimates. In a domain with smooth boundaries, the pair $(u,\om)$ is a strong solution of the transport equation \eqref{euler}, which  gives us the classical estimates for the vorticity: 
\begin{itemize}
\item $\|\om(t,\cdot)\|_{L^p(\Pi)}=\|\om_0\|_{L^p(\R^2)}$  for $p\in[1,+\infty]$;
\item if $\om_0$ is compactly supported in $B(0,R)$, then there exists $C>0$ such that $\om(t,\cdot)$ is compactly supported in $B(0,R+Ct)$;
\item for any $t>0$, we have $\int_{\Pi} \om(t,x)\, dx =   \int_{\Pi} \om_0(x)\, dx$.
\end{itemize}
The goal of this subsection is to prove such properties for the weak solution $(u^\e,\om^\e)$ defined in Definition \ref{sol-curve}. The main point is to remark that this pair is a renormalized solution in the sense of DiPerna and Lions (see \cite{dip-li}) of the transport equation.

Let us assume that $\om_0$ is $L^\infty$ and compactly supported in $B(0,R)\cap B(0,r)$. Moreover we fix $\e$ small enough such that the support of $\om_0$ does not intersect $\G_{\e}$ (any $\e\leq r$). We consider equation \eqref{transport} as a linear transport equation with given velocity field $u^\e$. Our purpose here is to show that if $\om^\e$ solves this linear equation, then so does $\beta(\om^\e)$ for a suitable smooth function $\b$. This follows from the theory developed in \cite{dip-li} (see also \cite{dej} for more details), where they need that the velocity field belongs to $L_{\loc}^1\left(\R^+,W_{\loc}^{1,1}(\RR)\right)$. Let us check that we are in this setting.

\begin{lemma}
Let $(u^\e,\om^\e)$ be a weak-solution of the Euler equations outside the curve $\G_\e$. As $\om^\e \in L^\infty (L^1\cap L^\infty)$ then 
\[u^\e \in L^\infty \left(\R^+,W_{\loc}^{1,1}(\RR)\right).\]
\end{lemma}

\begin{proof}
Here, we are not looking for estimates uniformly in $\e$, as later (e.g. Lemma \ref{I_est}). Then, we fix $\e>0$, and we rewrite \eqref{biot}:
\begin{eqnarray*} 
u^\e(x)&=&\frac{1}{2\pi} DT_\e^t(x) \Bigl( \int_{\Pi_\e} \Bigl(\frac{T_\e(x)-T_\e(y)}{|T_\e(x)-T_\e(y)|^2}- \frac{T_\e(x)-T_\e(y)^*}{|T_\e(x)-T_\e(y)^*|^2}\Bigl)^\perp \om^\e(y)\, dy + \a \frac{T_\e(x)^\perp}{|T_\e(x)|^2}\Bigl)\\
&:= & \frac{1}{2\pi} DT_\e^t(x) f(T_\e(x))
\end{eqnarray*}
where $\a$ is bounded by $\g + \| \om^\e\|_{L^\infty(L^1)}$.

We start by treating $f$. We change variable $\y=T_\e(y)$, and we obtain
\begin{eqnarray*}
f(z)&=& \int_{B(0,1)^c} \Bigl(\frac{z-\y}{|z-\y|^2}- \frac{z-\y^*}{|z-\y^*|^2}\Bigl)^\perp \om^\e(T_\e^{-1}(\y)) |\det DT_\e^{-1}(\y)| \, d\y + \a \frac{z^\perp}{|z|^2} \\
&=&\int_{B(0,1)^c}\frac{(z-\y)^\perp}{|z-\y|^2} g(\y) \, d\y - \int_{B(0,2)^c}\frac{(z-\y^*)^\perp}{|z-\y^*|^2} g(\y) \, d\y \\
&&- \int_{B(0,2)\setminus B(0,1)}\frac{(z-\y^*)^\perp}{|z-\y^*|^2} g(\y) \, d\y + \a \frac{z^\perp}{|z|^2}\\
&:=& f_1(z)-f_2(z)-f_3(z)+f_4(z),
\end{eqnarray*} 
with $g(\y)=\om^\e(T_\e^{-1}(\y)) |\det DT_\e^{-1}(\y)|$ belongs\footnote{This estimate is not uniform in $\e$.} to $L^\infty(L^1\cap L^\infty(\R^2))$. As $|z|= |T_\e(x)|\geq 1$, we are looking for estimates in $D^c$. Obviously we have that 
\[ f_4 \text{ belongs to } L^{\infty}(D^c) \text{ and } Df_4 \text{ belongs to } L^{\infty}(D^c).\]

Concerning $f_1$, we introduce $g_1:= g \chi_{D^c}$ where $\chi_S$ denotes the characteristic function on $S$. Hence
\[ f_1(z)=\int_{\R^2}\frac{(z-\y)^\perp}{|z-\y|^2}g_1(\y) \, d\y \text{ with } g_1\in L^\infty(L^1(\R^2)\cap L^\infty(\R^2)).\]
The standard estimates on Biot-Savart kernel in $\R^2$ (see e.g. Lemma \ref{ift}) and Calderon-Zygmund inequality
give that
\[ f_1 \text{ belongs to } L^\infty(\R^+\times \R^2) \text{ and } Df_1 \text{ belongs to } L^\infty(L^p(\R^2)),\ \forall p\in [1,\infty) .\]

For $f_2$, we can remark that for any $\y \in B(0,2)^c$ we have $|z-\y^*|\geq \frac12$. Therefore, the function $(z,\y)\mapsto \frac{(z-\y^*)^\perp}{|z-\y^*|^2} $ is smooth in $B(0,1)^c\times B(0,2)^c$, which gives us, by a classical integration theorem, that
\[ f_2 \text{ belongs to } L^\infty(\R^+\times D^c) \text{ and } Df_2 \text{ belongs to } L^\infty(\R^+\times D^c).\]

To treat the last term, we change variables $\theta = \y^*$
\[ f_3(z) = \int_{B(0,1)\setminus B(0,1/2)} \frac{(z-\theta)^\perp}{|z-\theta|^2} g(\theta^*) \frac{d\theta}{|\theta|^4}:= \int_{\R^2} \frac{(z-\theta)^\perp}{|z-\theta|^2} g_3(\theta)\, d\theta,\]
with $g_3(\theta)=\displaystyle \frac{g(\theta^*)}{|\theta|^4} \chi_{B(0,1)\setminus B(0,1/2)}(\theta)$ which belongs to $L^\infty(L^1(\R^2)\cap L^\infty(\R^2))$. Therefore, standard estimates on Biot-Savart kernel and Calderon-Zygmund inequality
give that
\[ f_3 \text{ belongs to } L^\infty(\R^+\times \R^2) \text{ and } Df_3 \text{ belongs to } L^\infty(L^p(\R^2)),\ \forall p\in [1,\infty) .\]

Now, we come back to $u^\e$. As $u^\e (x) = \frac{1}{2\pi} DT_\e^t(x) f(T_\e(x))$, with $DT_\e$ belonging to $L^p_{\loc}(\R^2)$ for $p<4$ and $f\circ T_\e$ uniformly bounded, we have that
\[ u^\e \text{ belongs to } L^{\infty}(\R^+;L^1_{\loc} (\R^2)).\]

Moreover, we have
\begin{eqnarray*}
 |Du^\e(x)| &\leq&  \frac{1}{2\pi}\Bigl( |D^2 T_\e(x)| |f(T_\e(x))| + |DT_\e(x)|^2 |(Df_1-Df_3) (T_\e(x))|\\
 &&+|DT_\e(x)|^2 |(-Df_2+Df_4) (T_\e(x))|\Bigl).
 \end{eqnarray*}
 We see that the second right hand side term belongs to $L^{\infty}(\R^+;L^1_{\loc} (\R^2))$ because $DT_\e$ belongs to $L^{3}_{\loc}(\R^2)$ and $(Df_1-Df_3)(T_\e(x))$ belongs to $L^{\infty}(\R^+;L^3 (\R^2))$. Similarly, the third right hand side term belongs to $L^{\infty}(\R^+;L^1_{\loc} (\R^2))$ because $DT_\e$ belongs to $L^{2}_{\loc}(\R^2)$ and $-Df_2+Df_4$ belongs to $L^{\infty}(\R^+\times D^c)$.
 
For the first right hand side term, we know that $f\circ T_\e$ is uniformly bounded, then we have to prove that $D^2 T_\e$ belongs to $L^1_{\loc}(\R^2)$ in order to finish the proof. Keeping in mind that $DT_\e$ is smooth everywhere, except near the end-points where it behaves like the inverse of the square root of the distance, and noting that the map $x\mapsto 1/\sqrt{|x|}$ belongs to $W^{1,1}_{\loc}(\R^2)$, it is natural to think that it holds true. However, this argument needs an estimate on $D^2 T$ up to the boundary. We have to check carefully in the proof of Proposition \ref{2.2} (see \cite{lac_euler}) how to gain this control. Actually, we can add a point at Proposition \ref{2.2}:
\begin{itemize}
\item$D^2 T$ extends continuously up to $\G$ with different values on each side of $\G$, except at the endpoints of the curve where $D^2 T$ behaves like the inverse of the power $3/2$ of the distance,
\end{itemize}
which implies that $D^2 T$ is bounded in $L^p(\Pi\cap B(0,R))$ for all $p<4/3$ and $R>0$. This extension allows us to finish the proof of this lemma.

For completeness, this extension of Proposition \ref{2.2} is proved in Annexe. 
\end{proof}

Therefore, \cite{dip-li,dej} imply that $\om^\e$ is a renormalized solution.

\begin{lemma}
 \label{renorm1}
Let $\om^\e$ be a solution of \eqref{transport}. Let $\beta:\R \rightarrow \R$ be a smooth function such that
\begin{equation*}
|\beta'(t)|\leq C(1+ |t|^p),\qquad \forall t\in \R,
\end{equation*}
for some $p\geq 0$. Then for all test function $\psi \in \mathcal{D}(\R^+ \times \RR)$, we have
\begin{equation*}
\frac{d}{dt}\int_{\RR} \psi \beta(\om^\e)\,dx=\int_{\RR} \beta(\om^\e) (\dt \psi +u^\e\cdot \nabla \psi)\,dx \:\: \textrm{in }\: L_{\loc}^1(\R^+).
\end{equation*}
\end{lemma}

Now, we write a remark from \cite{lac_miot}, in order to establish the some desired properties for $\om^\e$. 

\begin{remark}
\label{remark : conserv} (1) Lemma \ref{renorm1} actually still
holds when $\psi$ is smooth, bounded and has bounded first
derivatives in time and space. In this case, we have to consider
smooth functions $\beta$ which in addition satisfy $\beta(0)=0$, so
that $\beta(\om^\e)$ is integrable.
 This may be proved by approximating $\psi$ by smooth and compactly supported functions $\psi_n$ for which Lemma \ref{renorm1} applies, and by letting then $n$ go to $+\infty$.\\
(2) We apply the point (1) for $\beta(t)=t$ and $\p\equiv 1$, which gives
\begin{equation}\label{om-est-1}
 \int_{\R^2} \om^\e(t,x)\, dx =   \int_{\R^2} \om_0(x)\, dx \text{ for all }t>0.
 \end{equation} 
(3) We let $1\leq p<+\infty$. Approximating $\beta(t)=|t|^p$ by
smooth functions and choosing $\psi\equiv 1$ in Lemma \ref{renorm1},
we deduce that for an solution $\om^\e$ to \eqref{transport}, the
maps $t\mapsto \|\om^\e(t)\|_{L^p(\RR)}$ are continuous and constant.
In particular, we have
\begin{equation}\label{om-est-2}
 \|\om^\e(t)\|_{L^1(\RR)}+\|\om^\e(t)\|_{L^\infty(\RR)}:= \|\om_0\|_{L^1(\RR)}+\|\om_0\|_{L^\infty(\RR)}.
\end{equation}
\end{remark}

Unfortunately, we cannot establish now that $\om^\e$ is compactly supported uniformly in $\e$. For that, we need some estimates on the velocity.

\subsection{Velocity estimate}\ 

The goal of this subsection is to find a velocity estimate uniformly in $\e$, thanks to the explicit formula of $u^{\e}$ in function of $\om^{\e}$ and $\g$ (Biot-Savart law). We will need the following lemma from \cite{ift}:
\begin{lemma}\label{ift}
Let $S\subset\R^2$ and $g:S\to\R^+$ be a function belonging in $L^1(S)\cap L^p(S)$, for $p\in (2;+\infty]$. Then
\[\int_S \frac{g(y)}{|x-y|}dy\leq C\|g\|_{L^1(S)}^{\frac{p-2}{2(p-1)}}\|g\|_{L^p(S)}^{\frac{p}{2(p-1)}}.\]
\end{lemma}

For $h: \Pi_\e \to\R$ a function belonging in $L^1(S)\cap L^p(S)$, with $p\in (2;+\infty]$, we introduce
\begin{equation*}
I_1^\e[h](x)= \int_{\Pi_{\e}}\dfrac{(T_{\e}(x)-T_{\e}(y))^\perp}{|T_{\e}(x)-T_{\e}(y)|^2} h(y)dy,
\end{equation*}
and
\begin{equation*}
I_2^\e[h](x)=\int_{\Pi_{\e}}\dfrac{(T_{\e}(x)- T_{\e}(y)^*)^\perp}{|T_{\e}(x)- T_{\e}(y)^*|^2} h(y)dy.
\end{equation*}

Therefore, the Biot-Savart law can be written
\begin{equation}
\label{u_e}
u^{\e}(t,x)=\dfrac{1}{2\pi}DT_{\e}^t(x)(I_1^\e[\om^\e(t,\cdot)](x) - I_2^\e[\om^\e(t,\cdot)])(x)+(\g+m) H_{\e}(x),
\end{equation} 
with $m:=\int \om^\e(t,\cdot)= \int \om_0$ by \eqref{om-est-1}.
In \cite{lac_euler}, we manage to estimate directly $I_1^\e[\om^\e(t,\cdot)](x)$ and $I_2[\om^\e(t,\cdot)](x)$, uniformly in $x$ and $\e$, by $\| \om^\e(t,\cdot)\|_{L^1\cap L^\infty}$. In \cite{ift_lop_euler}, the authors obtain a uniform estimate in $x$ and $\e$ of
\[ \dfrac{1}{2\pi}DT_{\e}^t(x) I_1^\e[\om^\e(t,\cdot)](x) \text{ and }- \dfrac{1}{2\pi}DT_{\e}^t(x) I_2^\e[\om^\e(t,\cdot)] (x)+ m H_{\e}(x),\]
by $\| \om^\e(t,\cdot)\|_{L^1\cap L^\infty}$. The advantage of this decomposition is that each term has zero circulation around the small obstacle. Later, they have to study independently the last part of the velocity $\g H_{\e}$. As the size of the curve tends to zero, we see here that we have to use the decomposition from \cite{ift_lop_euler}. Then, let us introduce
\[ \tilde I_2^\e[h](x)=-I_2^\e[h](x) + m_h \frac{T_\e(x)^\perp}{|T_\e(x)|^2},\]
with $m_h=\displaystyle \int_{\Pi^\e} h(y) \, dy.$

\begin{lemma}\label{I_est} For any $p\in (2,\infty]$, there exists a constant $C_p>0$ depending only on the shape of $\G$, such that
\[ |I_{1}^\e[h](x)| \leq C \e \| h \|_{L^1}^{\frac{p-2}{2(p-1)}} \| h \|_{L^p}^{\frac{p}{2(p-1)}} \text{\ and\ } | \tilde I_2^\e[h](x) | \leq C \e \| h \|_{L^1}^{\frac{p-2}{2(p-1)}} \| h \|_{L^p}^{\frac{p}{2(p-1)}},\]
for all $x\in \R^2$, $\e>0$.
\end{lemma}

\begin{proof} The proof is the same as \cite{ift_lop_euler}, where you replace $DT_{\e}(x)$ by $1/\e$ and where you use Lemma \ref{ift}. For sake of clarity, we write the details.

We start by treating $I_1^\e$:
\[|I_1^\e[h](x) | \leq  \int_{\Pi_\e}\dfrac{|h(y)|}{|T(x/\e)-T(y/\e)|} dy.\]
We introduce $J=J(\x):= |\det(DT^{-1})(\x)|$ and $z=\e T(x/\e)$. Changing the variables $\y=\e T(y/\e)$, we find
\[|I_1^\e[h] (x)|\leq \e \int_{|\y|\geq\e}\dfrac{|h(\e T^{-1}(\y/\e))| J(\y/\e)}{|z-\y|}\, d\y.\]
Then, we denote $f^\e(\y)=|h(\e T^{-1}(\y/\e))| J(\y/\e)\h_{|\y\geq\e}$, with $\h_E$ the characteristic function of the set $E$. Changing variables back, we remark that 
\[ \|f^\e\|_{L^1(\R^2)}=\|h\|_{L^1},\]
and
\[ \|f^\e\|_{L^p(\R^2)}\leq C_p\|h\|_{L^p},\]
using the second point of Proposition \ref{2.2}. Now, we can use Lemma \ref{ift} to state that
\[ | I_1^\e[h](x) |\leq \e \int_{\R^2}\frac{f^\e(\y) }{|z-\y|}\, d\y \leq C_p\e \|f^\e\|_{L^1}^\frac{p-2}{2(p-1)} \|f^\e\|_{L^p}^\frac{p}{2(p-1)},\]
which allows us to conclude for $I_1^\e$.

Let us focus now on the second term: 
\[ \tilde I_2^\e[h](x)  = \int_{\Pi_\e}\left(-\dfrac{(T(x/\e)-T(y/\e)^*)^\perp}{|T(x/\e)-T(y/\e)^*|^2}+\dfrac{T(x/\e)^\perp}{|T(x/\e)|^2}\right) h(y)\, dy.\]
We use, as before, the notations $J$, $z$ and the change of variables $\y$
\begin{eqnarray*}
\tilde I_2^\e[h](x)  &=&  \e \int_{|\y|\geq \e} \left(-\dfrac{z-\e^2 \y^*}{|z-\e^2\y^*|^2}+\dfrac{z}{|z|^2} \right) h(\e T^{-1}(\y/\e)) J(\y/\e)\, d\y \\
|\tilde I_2^\e[h](x) | &\leq& \e \int_{|\y|\geq \e} \frac{\e^2|\y^*|}{|z||z-\e^2\y^*|}|h(\e T^{-1}(\y/\e))|J(\y/\e)d\y.
\end{eqnarray*}
using \eqref{frac}. As $z=\e T(x/\e)$, we have $|z|\geq \e$, hence 
\[ |\tilde I_2^\e[h](x) | \leq \e \int_{|\y|\geq \e} \frac{\e|\y^*|}{|z-\e^2\y^*|}|h(\e T^{-1}(\y/\e))|J(\y/\e)d\y.\]
Next, we change variables $\th=\e\y^*$, and we obtain:
\begin{eqnarray*}
|\tilde I_2^\e[h](x) | &\leq& \e \int_{|\th|\leq 1}\frac{|\th|}{|z-\e\th|}|h(\e T^{-1}(\th^*))|J(\th^*)\frac{\e^2}{|\th|^4}d\th \\
&\leq& \e \left(\int_{|\th|\leq 1/2}+\int_{1/2\leq|\th|\leq 1}\right):= \e(I_{21}+I_{22}).
\end{eqnarray*}

We start with $I_{21}$. If $|\th|\leq 1/2$ then $|z-\e\th|\geq\e/2$. Hence 
\begin{eqnarray*}
I_{21} &\leq&  \int_{|\th|\leq 1/2}2\e|\th||h(\e T^{-1}(\th^*))|J(\th^*)\frac{d\th}{|\th|^4} \\
&=& 2 \int_{|\y|\geq 2\e}\frac{|h(\e T^{-1}(\y/\e))|J(\y/\e)}{|\y|}d\y\leq 2\int_{\R^2} \frac{f^\e(\y)}{|\y|}d\y,
\end{eqnarray*}
with $f^\e$ defined above. Using again Lemma \ref{ift}, we can conclude for $I_{21}$.

To treat $I_{22}$, we put $g^\e(\th)=|g(\e T^{-1}(\th^*))|J(\th^*)\frac{\e^2}{|\th|^4}$. We have
\[I_{22}=\int_{1/2\leq|\th|\leq 1}\frac{|\th|}{|z-\e\th|}g^\e(\th)d\th .\]
Changing variables back, we remark that
\[ \|g^\e\|_{L^1(1/2\leq|\th|\leq 1)}\leq \|h\|_{L^1}.\]
Moreover, it is easy to see that
\[ \|g^\e\|_{L^p(1/2\leq|\th|\leq 1)} \leq C \e^{\frac{2p-2}{p}}\|h\|_{L^p}.\]
Next, we apply Lemma \ref{ift} with $g^\e$:
\begin{eqnarray*}
I_{22}&=&\frac{1}{\e}\int_{1/2\leq|\th|\leq 1} \frac{|\th|}{|z/\e-\th|}g^\e(\th)d\th \\
&\leq& \frac{C}{\e} \|g^\e\|_{L^1}^{\frac{p-2}{2(p-1)}}\|g^\e\|_{L^p}^{\frac{p}{2(p-1)}} \leq C_1  \|h \|_{L^1}^{\frac{p-2}{2(p-1)}}\|h\|_{L^p}^{\frac{p}{2(p-1)}},
\end{eqnarray*}
which ends the proof.
\end{proof}
In \cite{ift_lop_euler}, the authors use the estimate of $\frac{1}{2\pi}DT_{\e}^t I_1^\e[h]$ and $\frac{1}{2\pi}DT_{\e}^t \tilde I_2^\e[h]$ with $h=\om^\e(t,\cdot)$, and with $h= \zeta \cdot \na \F^\e$ ($\F^\e$ denoting the cutoff function of an $\e$ neighborhood of $\OM_\e$, see the proof of Corollary 4.1 therein), where there exist some $L^1$ and $L^\infty$ estimates for these two functions. In our case, we have again that $\om^\e(t,\cdot)$ are uniformly bounded in $L^1\cap L^\infty$, but we will only obtain $L^p$ estimates for $\na \F^\e$, with $p<4$ (see Lemma \ref{4.4}). It explains why we have to establish estimates for $h$ belonging in $L^1\cap L^p$ for $p\in (2,\infty]$.

Using the previous lemma with $h=\om^\e(t,\cdot)$, $p=+\infty$, and thanks to \eqref{u_e}, \eqref{T_Lp}, \eqref{om-est-1}, \eqref{om-est-2} we can deduce directly the following theorem:

\begin{theorem}\label{4.2} We denote $v^{\e}:= u^{\e} - \g H_{\e}$. For any $p<4$,  $v^{\e}$  is bounded in $L^\infty(\R^+,L^p_{\loc}(\Pi_{\e}))$ independently  of ${\e}$. More precisely, there exists a constant $C_p>0$ depending only on the shape of $\G$ and the initial conditions $\|\om_0\|_{L^1}$, $\|\om_0\|_{L^\infty}$, such that 
$$\|v^{\e}(t,\cdot)\|_{L^p(B(0,R) \cap \Pi_{\e})}\leq C_p(1+R^{2/p}), \text{\ for all\ } R>0, \ t\geq 0. $$
\end{theorem}

The difference with \cite{lac_euler} is that we have an estimate $L^p_{\loc}$ only on $v^{\e}$, then we will have to study independently $H_{\e}$. We note also that we cannot obtain $L^\infty$ estimates, and we have to check carefully that we can adapt the tools used in \cite{ift_lop_euler}.

\subsection{Compact support of the vorticity}\

Specifying our choice for $\beta$ in Lemma \ref{renorm1}, we are led to the
following.

\begin{proposition}
\label{compact_vorticity} Let $\om^\e$ be a weak solution of
\eqref{transport} such that
\begin{equation*}
\om_0  \text{ is compactly supported in } B(0,R_0)
\end{equation*}
for some positive $R_0$. Then there exists $C>0$ independent of $\e$ such that
\begin{equation*}
\om^\e(t,\cdot)  \text{ is compactly supported in } B(0,R_0+Ct),
\end{equation*}
for any $t\geq 0$.
\end{proposition}

\begin{proof} The main computation of this proof can be found in \cite{lac_miot}, but we have to write the details because the velocity has a different form and that we need that $C$ is independent of $\e$. We set $\beta(t)=t^2$ and use Lemma \ref{renorm1} with this choice. Let $\F \in \mathcal{D}(\R^+\times \RR)$. We claim that for all $T$
\begin{equation*}
\int_{\RR} \F(T,x) (\om^\e)^2(T,x)\,dx - \int_{\RR} \F(0,x) (\om^\e)^2(0,x)\,dx 
=\int_0^T\int_{\RR} (\om^\e)^2 (\dt \F +u^\e\cdot \nabla \F)\,dx \,
dt.
\end{equation*}
This is actually an improvement of Lemma \ref{renorm1}, in which the
equality holds in $L^1_{\loc}(\R^+)$.  Indeed, we have $\pd_t \om^\e
=-\diver (u^\e\om^\e)$ (in the sense of distributions) with $\om^\e\in
L^\infty$ and $u^\e \in L^\infty(\R^+,L^q_{\loc}(\RR))$ for all $q<4$,
which implies that $\pd_t \om^\e$ is bounded in
$L^1_{\loc}(\R^+,W^{-1,q}_{\loc}(\RR))$. Hence, $\om^\e$ belongs to
$C(\R^+,W^{-1,q}_{\loc}(\RR))\subset C_{w}(\R^+, L^2_{\loc}(\R^2))$, where $C_{w}(L_{\loc}^{2})$ stands for the space of maps $f$ such that for any sequence $t_n\to t$, the sequence $f(t_n)$ converges to $f(t)$ weakly in $L^2_{\loc}$. Since
on the other hand $t \mapsto \|\om^\e(t)\|_{L^2}$ is continuous by
Remark \ref{remark : conserv}, we have $\om^\e \in C(\R^+,L^2(\RR))$. Therefore the previous integral equality holds for all $T$.

Now, we choose a good test function. We let $\F_0$ be a non-decreasing function on $\R$, which
is equal to $1$ for $s\geq 2$ and vanishes for $s\leq 1$ and we
set $\F(t,x)=\F_0(|x|/R(t))$, with $R(t)$ a smooth, positive
and increasing function to be determined later on, such that
$R(0)=R_0$. For this choice of $\F$, we
have $(\om_0(x))^2 \F(0,x)\equiv 0$.

We compute then
\begin{equation*}
\na \F= \frac{x}{|x|}\frac{\F_0'}{R(t)}
\end{equation*}
and
\begin{equation*}
\pd_t \F = -\frac{R'(t)}{R^2(t)}|x|\F_0'.
\end{equation*}
We obtain
\begin{eqnarray*}
\int_{\RR} \F(T,x) (\om^\e)^2(T,x)\,dx & =&\int_0^T \int_{\RR} (\om^\e)^2 \frac{\F_0'(\frac{|x|}{R})}{R}\Bigl( u^\e(x) \cdot \frac{x}{|x|}-\frac{R'}{R}|x|\Bigl)\, dx\, dt\\
&\leq& \int_0^T \int_{\RR} (\om^\e)^2 \frac{|\F_0'|(\frac{|x|}{R})}{R} (C -R')\, dx\, dt,
\end{eqnarray*}
where $C$ is independent of $\e$. Indeed, we have that 
\[ u^{\e}(t,x)=\dfrac{1}{2\pi}DT_{\e}^t(x)(I_1^\e + \tilde I_2^\e +\g \frac{T_\e(x)^\perp}{|T_\e(x)|^2} )\]
with  $ |I_1^\e + \tilde I_2^\e| \leq C_1 \e$ (see Lemma \ref{I_est}) and $DT_\e(x) = \frac{1}{\e} DT(\frac{x}\e)$. Using Remark \ref{2.5}, we know that there exist some positives $C_3,C_4$ independent of $\e$, such that
\[ |DT(\frac{x}\e) | \leq C_2 |\b| \text{ and } C_4 |\b| \frac{|x|}\e \leq |T(\frac{x}\e)|,\]
for all $|x|\geq R_0$. Putting together all these inequalities, we obtain $C=\frac{1}{2\pi}C_2( |\b| C_1 + \frac{|\g|}{R_0 C_4})$.
Taking $R(t) =R_0 + Ct$, we arrive at
\begin{equation*}
\int_{\RR} \F(T,x) (\om^\e)^2(T,x)\,dx\leq 0,
\end{equation*}
which ends the proof.
\end{proof}

\begin{remark}  We only use in this paper that $(u^\e,\om^\e)$ is a weak solution of the Euler equations outside the curve (see Definition \ref{sol-curve}). If the uniqueness is proved, we could simplify the proofs of \eqref{om-est-1}, \eqref{om-est-2} and Proposition \ref{compact_vorticity}. Indeed, we would say by uniqueness that $\om^\e$ is the weak-$*$ limit of $\F^{\e,\y} \om^{\e,\y}$ with $\OM^{\e,\y}$ defined in Subsection \ref{thicken}. As $ \om^{\e,\y}$ verifies the transport equation in a strong sense, we have:
\begin{itemize}
\item for all $\y$ and $t$, $\| \om^{\e,\y}(t,\cdot) \|_{L^1\cap L^\infty} = \| \om_0 \|_{L^1\cap L^\infty}$, which means that $\| \om^{\e}(t,\cdot) \|_{L^1\cap L^\infty} \leq \| \om_0 \|_{L^1\cap L^\infty}$
which is sufficient;
\item $\om^\e$ is also the weak-$*$ limit of $\h_{\Pi_{\e,\y}} \om^{\e,\y}$, and $\int \h_{\Pi_{\e,\y}} \om^{\e,\y} = \int \om_0$ for all $t$, so we obtain \eqref{om-est-1};
\item it is easy to prove that there exists $C$ independent of $\y$ and $\e$ such that  $ \om^{\e,\y}(t,\cdot)$ is compactly supported in $B(0,R_1+Ct)$, which proves Proposition \ref{compact_vorticity}, using test functions supported in $B(0,R_1+Ct)^c$.
\end{itemize}
\end{remark}

\subsection{Cutoff function}\label{cutoff}\

The function $u^{\e}$ is defined on $\R^2$, but we prefer to multiply it by an ${\e}$-dependent cutoff function for a neighborhood of $\OM_{\e}$. Indeed, $\curl u^\e = \om^\e + g_{\om^\e} \d_{\G_\e}$, so the cutoff function allows us to remove the dirac mass and the jump of the velocity through the curve.

Let $\F\in C^\infty(\R)$ be a non-decreasing function such that $0\leq\F\leq 1$, $\F(s)=1$ if $s\geq 3$ and $\F(s)=0$ if $s\leq 2$. Then we introduce 
$$\F^{\e}=\F^{\e}(x)=\F(|T_{\e}(x)|).$$
Clearly $\F^{\e}$ is $C^\infty(\R^2)$ vanishing in a neighborhood of $\overline{\OM_{\e}}$.

We require some properties of $\na\F^{\e}$ which we collect in the following Lemma.

\begin{lemma}\label{4.4} The function $\F^{\e}$ defined above has the following properties:
\begin{itemize}
\item[(a)] $H_{\e}\cdot \na\F^{\e}\equiv 0$ in $\Pi_{\e}$,
\item[(b)] there exists a constant $C>0$ such that the Lebesgue measure of the support of $\F^{\e}-1$ is bounded by $C{\e^2}$.
\item[(c)] for all $p<4$, there exists a constant $C_p>0$ such that $\|\na \F^{\e} \|_{L^p}\leq \e^{\frac2p -1} C_p$.
\end{itemize}
\end{lemma}

\begin{proof}
First, we remark that
\[H_{\e}(x)=\frac{1}{2\pi}\na^\perp \ln |T_{\e}(x)|=\frac{1}{2\pi|T_{\e}(x)|}\na^\perp |T_{\e}(x)|,\]
and
\begin{equation*}
\na \F^{\e}=\F'(|T_{\e}(x)|)\na |T_{\e}(x)|
\end{equation*}
what gives us the first point.

Concerning the second point, the support of $\F^{\e} -1$ is contained in the subset $\{ x\in\Pi_{\e} | 1  \leq |T_{\e}(x)| \leq 3\}$. By Proposition \ref{biholo-est}, the Lebesgue measure can be estimated as follows:
\[\int_{ 1  \leq |T_{\e}(x)| \leq 3}dx=\int_{1  \leq |z| \leq 3} |\det(DT_{\e}^{-1})|(z) dz \leq C_1 {\e^2}.\]

Finally, we have
\[ |\na \F^{\e}(x)| \leq |\F'(|T_\e(x)|)| | DT_\e(x)|,\]
hence,
\[ \| \na \F^{\e} \|_{L^p}\leq C \|DT_\e(x) \|_{L^p(\{x| |T_\e(x)|\leq 3\})}.\]
Using, that $T(z)$ goes to infinity when $|z|\to \infty$, we can state that there exists $R_1>0$ such that $\{ y\in \R^2 | |T(y)|\leq 3\}=T^{-1}(B(0,3)\setminus B(0,1)) \subset B(0,R_1)$. 
We rewrite the computation made in the proof of Proposition \ref{biholo-est}:
\begin{eqnarray*}
 \Bigl( \int_{\{x | |T(x/\e)|\leq 3\}} \bigl| \frac1\e DT(\frac{x}\e) \bigl|^p\, dx\Bigl)^{1/p} &=& \e^{\frac2p - 1} \Bigl( \int_{\{y | |T(y)|\leq 3\}} \bigl|  DT(y) \bigl|^p\, dy\Bigl)^{1/p} \\
&\leq& \e^{\frac2p - 1} \Bigl( \int_{B(0,R_1)} \bigl|  DT(y) \bigl|^p\, dy\Bigl)^{1/p} \\
&\leq& \e^{\frac2p - 1} C_p,
\end{eqnarray*}
which ends the proof.
\end{proof}

\begin{remark}\label{rk-phi} As $v^\e(x)=\frac{1}{2\pi} DT_\e(x) (I_1+\tilde I_2)$, using Lemma \ref{I_est} and the proof of point (c), we can state that
for all $p<4$, there exists a constant $C_p>0$ such that 
\[\|  v^\e(x)  \|_{L^p(\{x| |T_\e(x)|\leq 3\})}\leq \e^{\frac2p} C_p.\]
\end{remark}

In the case where the obstacle is smooth (see \cite{ift_lop_euler}), $DT$ is bounded, which implies that the norm $L^2$ of $\na \F^{\e}$ is bounded. Moreover, in their case, the part of velocity $v^{\e}$ is bounded independently of ${\e}$, so we can prove that the limits of  $v^{\e} \cdot  \na \F^{\e}$ and  $v^{\e} \cdot  \na^\perp \F^{\e}$ is bounded in $L^\infty(L^2_{\loc})$. As $L^2_{\loc}$ is compactly imbedded in $H^{-1}_{\loc}$, we can prove by Aubin-Lions Lemma that the divergence and the curl of $\F^\e v^\e$ is precompact in $C([0,T]; H^{-1}_{\loc} (\R^2))$. Finally the authors of \cite{ift_lop_euler} conclude thanks to the Div-Curl Lemma. 

In our case, let us show that we can apply this argument. We use Lemma \ref{4.4} and Remark \ref{rk-phi} with $p=3$, then
\begin{equation}\label{v-phi-1}
 \|v^{\e} \cdot  \na^\perp \F^{\e}\|_{L^{3/2}}\leq   \|v^\e(x) \|_{L^3(\supp (\na \F^{\e}))} \|\na \F^{\e} \|_{L^3(\supp (\na \F^{\e}))} \leq C\e^{1/3} .
\end{equation}
Similarly, we have 
\begin{equation}\label{v-phi-2}
 \|v^{\e} \cdot  \na \F^{\e}\|_{L^{3/2}}\leq   \|v^\e(x) \|_{L^3(\supp (\na \F^{\e}))} \|\na \F^{\e} \|_{L^3(\supp (\na \F^{\e}))} \leq C\e^{1/3} .
\end{equation}
As $H^1(\R^2)$ is imbedded in $L^{3}(\R^2)$, so $L^{3/2}(\R^2)$ is imbedded in $H^{-1}(\R^2)$, and we could apply Aubin-Lions Lemma. This last computation is an improvement of a naive estimate. Indeed, we would have written that: 
\[ \|v^{\e} \cdot  \na^\perp \F^{\e}\|_{L^{1}} \leq C\|v^{\e}\|_{L^{4}} \|DT_\e \|_{L^{4}_{loc}} \|1 \|_{L^{2}(\supp(\na \F^\e))} \leq C_1\frac1\e \e,\]
assuming that Theorem \ref{4.2} and point (ii) of Proposition \ref{biholo-est} could be applied for $p=4$, which is not true. Even in this limit case, we remark that we can only control the $L^1$ norm of $v^{\e} \cdot  \na^\perp \F^{\e}$, which does not embed in $H^{-1}$ in dimension two. With this estimate, the argument from \cite{ift_lop_euler} falls down. Estimate \eqref{v-phi-1} was established thanks to point (c) of Lemma \ref{4.4} and Remark \ref{rk-phi}. Without this improvement, we would have adapted the arguments from \cite{lac_euler}. However, we choose here to use techniques from \cite{ift_lop_euler}, because it is faster and it is less technical.

As we decompose $u^\e=v^\e+\g H_{\e}$, we have to focus on the harmonic part.

\begin{lemma} \label{H-limit}
Let $H:= x^\perp / (2\pi |x|^2)$ and fix $R>0$. Then,
\[  H_{\e} \to H,\]
strongly in $L^p(B(0,R))$ as $\e\to 0$, for any $p<2$.
\end{lemma}

\begin{proof}
The proof is similar than \cite{ift_lop_euler}, because there is the same behaviour at infinity (see Remark \ref{2.5}). However this lemma is stated in \cite{ift_lop_euler} only with $p=1$. We will see in the following subsection that we need for $p=3/2$. For this reason, we rewrite the proof here.

We fix $p<2$ and we decompose:
\begin{eqnarray*}
\| H_{\e}-H\|_{L^p(B(0,R))} &\leq&  \| H_{\e}-H\|_{L^p(B(0,R)\cap\{|T_\e(x)|\geq 2\})} + \| H_{\e}\|_{L^p(\{|T_\e(x)|\leq 2\})} + \| H\|_{L^p(\{|T_\e(x)|\leq 2\})}\\
&:=& \mathcal{I}_1+\mathcal{I}_2+\mathcal{I}_3.
\end{eqnarray*}
From the proof of Lemma \ref{4.4}, we know that the Lebesgue measure of the set $\{|T_\e(x)|\leq 2\}$ tends to zero as $\e\to 0$. Having in mind that $H$ belongs in $L^q_{\loc}$ for $q\in (p,2)$, we can state that $\mathcal{I}_3\to 0$ as $\e\to 0$.

Concerning $\mathcal{I}_2$, we change variables $y=x/\e$:
\begin{eqnarray*}
 \mathcal{I}_2 &=& \Bigl( \int_{\{|T(x/\e)|\leq 2\}} \Bigl| \frac{1}{2\e\pi} DT(x/\e)^t \frac{T(x/\e)^\perp}{|T(x/\e) |^2} \Bigl|^p\, dx\Bigl)^{1/p} \\
 &=&  \Bigl( \int_{\{|T(y)|\leq 2\}} \Bigl| \frac{1}{2\e\pi} DT(y)^t \frac{T(y)^\perp}{|T(y) |^2} \Bigl|^p \e^2\, dy\Bigl)^{1/p}\\
 &\leq& \frac{\e^{\frac{2-p}{p}}}{2\pi} \| DT \|_{L^p(\{|T(y)|\leq 2\})}
 \end{eqnarray*}
 which gives the result because $DT$ belongs to $L^q_{\loc}$ for $q<4$ (see Proposition \ref{2.2}).
 
 For $\mathcal{I}_1$, we use Remark \ref{2.5}: $T(y)=\b y+h(y)$, with $\b\in \R^*$, and $h$ holomorphic such that $|Dh(y)|\leq C/|y|^2$. Changing variables as above, we find:
 \begin{eqnarray*}
 \mathcal{I}_1 &=& \frac{\e^{\frac{2-p}{p}}}{2\pi} \Bigl( \int_{B(0,R/\e)\cap\{|T(y)|\geq 2\}} \Bigl| DT(y)^t \frac{T(y)^\perp}{|T(y) |^2}-\frac{y^\perp}{|y|^2} \Bigl|^p\, dy \Bigl)^{1/p} \\
 &=& \frac{\e^{\frac{2-p}{p}}}{2\pi} \Bigl( \int_{B(0,R/\e)\cap\{|T(y)|\geq 2\}} \Bigl| (\b \mathbb{I}+ Dh^t(y)) \frac{(\b y+h(y))^\perp}{|\b y+h(y) |^2}-\b \mathbb{I} \frac{\b y^\perp}{|\b y|^2} \Bigl|^p\, dy\Bigl)^{1/p} \\ 
 &\leq& \frac{\e^{\frac{2-p}{p}}}{2\pi} \Bigl( \int_{B(0,R/\e)\cap\{|T(y)|\geq 2\}} \Bigl|  Dh^t(y) \frac{(\b y+h(y))^\perp}{|\b y+h(y) |^2} \Bigl|^p\, dy\Bigl)^{1/p}\\
 &&+ \frac{\e^{\frac{2-p}{p}}}{2\pi} \Bigl( \int_{B(0,R/\e)\cap\{|T(y)|\geq 2\}} \Bigl| \b \mathbb{I} \Bigl( \frac{(\b y+h(y))^\perp}{|\b y+h(y) |^2}- \frac{\b y^\perp}{|\b y|^2}\Bigl) \Bigl|^p\, dy\Bigl)^{1/p}\\
 &\leq& C \e^{\frac{2-p}{p}} \Bigl( \int_{\{|T(y)|\geq 2\}} \frac{1}{|y|^{3p}}\, dy\Bigl)^{1/p}+
  C\e^{\frac{2-p}{p}} \Bigl( \int_{B(0,R/\e)\cap\{|T(y)|\geq 2\}} \Bigl(\frac{|h(y)|}{|y||\b y+h(y) |} \Bigl)^p\, dy\Bigl)^{1/p},
\end{eqnarray*}
using \eqref{frac}. If $p\in (1,2)$, we bound the right hand side term by 
\[ C_1 \e^{\frac{2-p}{p}}+ C_2  \e^{\frac{2-p}{p}} \Bigl( \int_{\{|T(y)|\geq 2\}} \frac{1}{|y|^{2p}}\, dy\Bigl)^{1/p}\leq C_3 \e^{\frac{2-p}{p}}\]
which tends to zero if $\e\to 0$.

If $p=1$
we bound the right hand side term by 
\[ C_1 \e+ C_2 \e \ln (R/\e)\]
which also tends to zero if $\e\to 0$.
\end{proof}

Now, we need some estimates of $\om^\e_t$ and $v^\e_t$ in order to use Aubin-Lions Lemma.

\subsection{Temporal estimates}\

Although in our case, the vorticity equation \eqref{transport} is verify in the sense of distribution, we directly see that it also means that $\om^\e_t$ is bounded in $L^\infty([0,T]; W^{-1,1}_{\loc}(\R^2))$. Indeed, we have proved that $v^\e$ and $H_{\e}$ are bounded in $L^\infty(L^1_{\loc})$, whereas $\om^\e$ is bounded in $L^\infty$. We recall from Proposition \ref{compact_vorticity}, that for $T$ fixed, there exists $R_1>0$ such that $\om^\e(t,\cdot)$ is compactly supported in $B(0,R_1)$ for all $0\leq t \leq T$ and $\e>0$. Additionally, $\om^\e_t$ is also compactly supported in the same ball.

Concerning $v^\e_t$, we have to prove that Proposition 4.1 and Corollary 4.1 of \cite{ift_lop_euler} hold true in our case. We introduce the stream function associated to $\om^\e$ by $\p^\e:= G_\e[\om^\e]$, with
\[G_\e[f](x) = \int_{\Pi_\e} G_\e(x,y) f(y)\, dy, \ \forall f\in C^\infty_c(\Pi_\e)\]
(see \eqref{green} for the explicit formula). We note that $K_\e[f]=\na^\perp G_\e[f]$.

\begin{proposition} For each $R,T>0$, there exists a constant $C>0$ independent of $\e$, such that
\[ \Bigl| \int_{\Pi_\e} \f(x)  \p^\e_t(t,x)\, dx \Bigl|\leq C(\| \f\|_{L^1}+\| \f\|_{L^1}^{1/4} \| \f\|_{L^3}^{3/4}),\]
for every $\f\in C_0(\Pi_\e\cap B(0,R))$ and for all $0\leq t\leq T$.
\end{proposition}

\begin{proof} We differentiate with respect of $t$ the stream function: $\p^\e_t=G_\e[\om^\e_t]$, which means that
\[ \D \p^\e_t = \om_t^\e \text{ in } \Pi_\e, \text{ and } \p_t^\e=0 \text{ on }\G_\e.\]
To obtain information on the behavior of $\p^\e_t$ at infinity, we use the same argument than \eqref{K-inf} to state that
\begin{equation}\label{psi}
 | \p_t^\e(t,x)-L[\om^\e_t(t,\cdot)](x)| = O(1/|x|) \text{ at infinity},
 \end{equation}
where the functional $L$ is defined by
\[ \zeta\mapsto L[\zeta]:= -\frac1{2\pi} \int_{\Pi_\e} \ln |T_\e(y)| \zeta(y)\, dy,\]
for any test function $\zeta$. The asymptotic behavior is not independent of $\e$, but we will only need that for $\e$ fixed.

Moreover, we recall that \eqref{K-inf} gives
\begin{equation}\label{na-psi}
 |\na \p_t^\e|=|K_\e[\om_t^\e]| =O(1/|x|^2) \text{ at infinity.}
 \end{equation}

Let $\f$ be a fixed test function in $C_0(\Pi_\e\cap B(0,R))$, we define
\[ \y:= G_\e[\f] +\frac{m_\f}{2\pi} \ln |T_\e|,\]
where $m_\f = \int_{\Pi_\e} \f(x)\, dx$. As above, we can remark that $\y$ satisfies
\[ \D \y = \f \text{ in } \Pi_\e, \text{ and } \y=0 \text{ on }\G_\e,\]
\begin{equation}\label{eta}
  \y(x)= \frac{m_\f}{2\pi} \ln |T_\e|(x) + L[\f](x) + O(1/|x|) \text{ at infinity}
\end{equation}
and
\begin{equation}\label{na-eta}
 |\na( \y - \frac{m_\f}{2\pi} \ln |T_\e|)|=|K_\e[\f]| =O(1/|x|^2) \text{ at infinity.}
\end{equation}

We compute
\begin{eqnarray*}
 \int_{\Pi_\e} \f(x) \p^\e_t(t,x)\, dx &=&  \int_{\Pi_\e} \D\y(x) \p^\e_t(t,x)\, dx\\
 &=&  \int_{\Pi_\e} \y(x) \D \p^\e_t(t,x)\, dx +  \int_{\pd \Pi_\e} ( \p^\e_t \na \y - \y \na \p^\e_t)\cdot \hat{n} \, ds\\
 &:=& I + J
 \end{eqnarray*}
where the boundary terms include the terms at infinity.

Using \eqref{transport}\footnote{this equality is given in $\mathcal{D}'(\R^+)$, but it holds for all $t$ (see the proof of Proposition \ref{compact_vorticity}).}, we begin by estimating $I$:
\[I= \int_{\Pi_\e} \y(x) \om^\e_t(t,x)\, dx= \int_{\Pi_\e} \na \y(x)\cdot (v^\e + \g H_{\e}) \om^\e \, dx,\]
then
\[ | I | \leq \| \na \y\|_{L^3(B(0,R_1))} \|v^\e + \g H_{\e}\|_{L^{3/2}(B(0,R_1))} \| \om^\e \|_{L^\infty} \leq C  \| \na \y\|_{L^3(B(0,R_1))},\]
thanks to \eqref{om-est-1} and using again Theorem \ref{4.2} and Lemma \ref{H-limit} with $p=3/2$. Moreover, as we have
\[ \na^\perp \y(x) =\frac1{2\pi} DT_\e^t(x) \bigl( I_1^\e[\f] - I_2^\e[\f] + m_\f \frac{T_\e(x)^\perp}{|T_\e(x)|^2 }\bigl),\]
we can use point (ii) of Proposition \ref{biholo-est} for $p=1$ and Lemma \ref{I_est} for $p=3$ to conclude that
\[ | I | \leq C\frac{1}\e \e \| \f\|_{L^1}^{1/4} \| \f\|_{L^3}^{3/4}.\]

Concerning the boundary terms $J$, we note that the integrals on $\G_\e$ vanish, because $\y=\p_t^\e = 0$ on the curve. Thanks to \eqref{na-psi} and \eqref{eta}, we have
\[  \int_{\pd B(0,R)} \y \na \p^\e_t\, ds \leq C \frac{\ln R}R\]
which tends to zero as $R\to \infty$. Using now \eqref{psi} and \eqref{na-eta}, we obtain
\[ |J| \leq C m_\f |L[\om^\e_t]|\leq C \| \f \|_{L^1}|L[\om^\e_t]| .\]
To finish the proof, we have to estimate $|L[\om^\e_t]|$. Keeping in mind that $H_{\e}(y)= \na^\perp (\ln |T_\e(y)|)$, we compute
\begin{eqnarray*}
L[\om^\e_t] &=& -\frac1{2\pi} \int_{\Pi_\e} \na(\ln |T_\e(y)|)\cdot (v^\e(t,y)+\g H_{\e}(y)) \om^\e(t,y)\, dy \\
&=& -\frac1{2\pi} \int_{\Pi_\e} H_{\e}(y)^\perp \cdot v^\e(t,y) \om^\e(t,y)\, dy \\
|L[\om^\e_t]| &\leq& \| H_{\e}(y) \|_{L^{3/2}(B(0,R_1))}  \| v^\e \|_{L^{3}(B(0,R_1))}  \| \om^\e \|_{L^{\infty}} \leq C
\end{eqnarray*}
using Theorem \ref{4.2} with $p=3$ and Lemma \ref{H-limit} with $p=3/2$.

Putting together the estimates concludes the proof.
\end{proof}

In \cite{ift_lop_euler}, it is sufficient to bound the integral by $ \| \f\|_{L^1}^{1/2} \| \f\|_{L^\infty}^{1/2}$. We will see in the following proposition that we need $\| \f\|_{L^p}$ for some $p<4$ instead of $\| \f\|_{L^\infty}$ (e.g. $p=3$). For this goal, we use in the previous proof Lemma \ref{4.2} for $h\in L^p$ instead of $L^\infty$, which justifies the extension for $p\neq \infty$ in Lemma \ref{4.2}. 

We also see at the end of the previous proof that we cannot write $\| v^\e \|_{L^\infty}$, so it explains why we need the extension for $p>1$ in Lemma \ref{H-limit}.

Thanks to this proposition, we can establish the main result of this subsection.  

\begin{corollary} \label{v_t}
Let $R,T>0$. Then there exists a constant $C=C(R,T)>0$ such that
\begin{equation*}
\|(\F^\e v^\e)_t(t,\cdot)\|_{H^{-3}(B(0,R))}\leq C,
\end{equation*}
for all $\e$ and $0\leq t\leq T$.
\end{corollary}

\begin{proof}
Let $\zeta \in (H_0^3(B(0,R)))^2$. Applying twice the previous proposition, we compute
\begin{eqnarray*}
| \langle \zeta, (\F^\e v^\e)_t (\cdot, t)\rangle | &= & \Bigl| \int \zeta \F^\e \na^\perp \p^\e_t(t,\cdot)\Bigl| = \Bigl| \int \curl(\zeta \F^\e) \p^\e_t(t,\cdot)\Bigl|\\
&=& \Bigl| \int \curl(\zeta) \F^\e \p^\e_t(t,\cdot) +  \int \zeta\cdot \na^\perp \F^\e \p^\e_t(t,\cdot)\Bigl| \\
&\leq& C(\| \curl(\zeta) \F^\e \|_{L^1}+\| \curl(\zeta) \F^\e \|_{L^1}^{1/4} \| \curl(\zeta) \F^\e \|_{L^3}^{3/4})\\
&& + C(\|  \zeta\cdot \na^\perp \F^\e \|_{L^1}+\| \zeta\cdot \na^\perp \F^\e \|_{L^1}^{1/4} \| \zeta\cdot \na^\perp \F^\e \|_{L^3}^{3/4})\\
&\leq & C (\| \curl \zeta\|_{L^\infty}+ \| \zeta\|_{L^\infty}),
\end{eqnarray*}
since $\| \na \F^\e \|_{L^1} \leq C \e$ and $\| \na \F^\e \|_{L^3} \leq C \e^{-1/3}$ (see point (c) of Lemma \ref{4.4}). Sobolev embedding theorem allows us to end the proof.
\end{proof}

We understand now why we need all these estimates in terms of $\| h \|_{L^3}$ instead of $\| h \|_{L^\infty}$. Indeed, we use them with $\na \F^\e$, and we remark in Lemma \ref{4.4} that we cannot obtain estimates in $L^p$ norm for $p\geq 4$, because of $DT$ which blows up at the end-points like the inverse of the square root of the distance.

\section{Passing to the limit}

Thanks to \eqref{v-phi-1}, \eqref{v-phi-2} and the previous corollary, we can exactly apply the arguments from \cite{ift_lop_euler}. In order to simplify the reading, we write the details.

\subsection{Strong compactness for the velocity}\

The principal tool is a parameterized version of Tartar and Murat's Div-Curl Lemma, whose proof can be found in \cite{div-curl}:
\begin{lemma}\label{div-curl}
Fix $T>0$ and let $\{F^\e(t,\cdot)\}$ and $\{G^\e(t,\cdot)\}$ be vector fields on $\R^2$ for $0\leq t\leq T$. Suppose that:
\begin{itemize}
\item[(a)] both $F^\e\to F$ and $G^\e \to G$ weak-$*$ in $L^\infty([0,T]; L^2_{\loc} (\R^2;\R^2))$ and also strongly in $C([0,T]; H^{-1}_{\loc} (\R^2;\R^2))$;
\item[(b)] $\{ \diver F^\e \}$ is precompact in $C([0,T]; H^{-1}_{\loc} (\R^2))$;
\item[(c)] $\{ \curl G^\e \}$ is precompact in $C([0,T]; H^{-1}_{\loc} (\R^2;\R))$.
\end{itemize}
Then $F^\e\cdot G^\e \rightharpoonup F\cdot G$ in $\mathcal{D}'([0,T]\times \R^n)$.
\end{lemma}

We will use the Div-Curl Lemma with $F^\e=G^\e=\F^\e v^\e$. For that, we check now that the three points of this lemma  are verified.

For point (a), we know from Theorem \ref{4.2} that $\{ \F^\e v^\e \}$ is bounded in $L^\infty([0,T]; L^2_{\loc} (\R^2))$. Moreover, thanks to Corollary \ref{v_t} we know that $\{ \F^\e v^\e \}$ is equicontinuous from $[0,T]$ to $H^{-3}_{\loc}$. Then we can apply Aubin-Lions Lemma (see \cite{temam}) to state that $\{ \F^\e v^\e \}$ is precompact in $C([0,T]; H^{-1}_{\loc} (\R^2))$. Passing to a subsequence if necessary, we conclude that there exists $v\in L^\infty([0,T]; L^2_{\loc}) \cap C([0,T]; H^{-1}_{\loc} )$ such that
\begin{equation*}
\F^\e v^\e \to v
\end{equation*}
weak-$*$ in $L^\infty([0,T]; L^2_{\loc})$ and strongly in $C([0,T]; H^{-1}_{\loc})$.

For point (b), we start by remarking that $(\diver (\F^\e v^\e))_t=\diver (\F^\e v^\e)_t$ is bounded in $L^\infty([0,T]; H^{-4}_{\loc})$ (see Corollary \ref{v_t}). Moreover, we know that
\[ \diver (\F^\e v^\e) = v^\e\cdot \na \F^\e\]
is bounded in $L^\infty([0,T]; L^{3/2})$ (see  \eqref{v-phi-2}). Since $L^{3/2}_{\loc}$ is compactly imbedded in $H^{-1}_{\loc}$, we can again apply Aubin-Lions Lemma to conclude that the divergence is precompact in $C([0,T]; H^{-1}_{\loc})$.

Finally, we do the same thing with the curl:
\begin{itemize}
\item $(\curl (\F^\e v^\e))_t=\curl (\F^\e v^\e)_t$ is bounded in $L^\infty([0,T]; H^{-4}_{\loc})$;
\item $\curl (\F^\e v^\e) = \F^\e \om^\e + v^\e\cdot \na^\perp \F^\e$ is bounded in $L^\infty([0,T]; L^{3/2})$
\end{itemize}
then  the curl is precompact in $C([0,T]; H^{-1}_{\loc})$.

Therefore, we can apply Lemma \ref{div-curl} to ensure that $| \F^\e v^\e|^2 \rightharpoonup |v|^2$ in $\mathcal{D}'$, which implies the following theorem.

\begin{theorem} \label{v-limit}
For all $T>0$, we can extract a subsequence $\e_k \to 0$ such that $\F^\e v^\e \to v$ strongly in $L^2_{\loc}([0,T]\times \R^2)$.
\end{theorem}
By a diagonal extraction, we have a subsequence $\e_k \to 0$ such that the convergence holds in $L^2_{\loc}(\R^+\times \R^2)$.

\subsection{The asymptotic vorticity equation}\ 

We begin by observing that the sequence $\{\F^\e\om^\e\}$ is bounded in $L^\infty(\R^+\times \R^2)$, then, passing to a subsequence if necessary, we have 
\[\F^\e\om^\e\rightharpoonup \om\text{, weak-$*$ in }L^\infty(\R^+\times \R^2).\]
We already have a limit velocity: $u:= v+\g H$.

The purpose of this section is to prove that $u$ and $\om$ verify, in an appropriate sense, the system:
\begin{equation}
\left\lbrace \begin{aligned}
\label{tour_equa}
&\pd_t \om+u\cdot \na\om=0, & \text{ in } (0,\infty) \times\R^2 \\
& \diver u=0 \text{ and }\curl u=\om+\g \d, &\text{ in }(0,\infty) \times\R^2  \\
& |u|\to 0, &\text{ as }|x|\to \infty \\
& \om(0,x)=\om_0(x), &\text{ in }\R^2.
\end{aligned} \right .
\end{equation}
where $\d$ is the Dirac function centered at the origin.

\begin{definition} The pair $(u,\om)$ is a weak solution of the previous system if
\begin{itemize}
\item[(a)] for any test function $\f\in C^\infty_c([0,\infty)\times\R^2)$ we have 
\[ \int_0^\infty\int_{\R^2}\f_t\om dxdt +\int_0^\infty \int_{\R^2}\na\f\cdot u\om dxdt+\int_{\R^2}\f(0,x)\om_0(x)dx=0,\]
\item[(b)] we have $\diver u=0$ and $\curl u=\om+\g \d$ in the sense of distributions of $\R^2$, with $|u|\to 0$ at infinity.
\end{itemize}
\end{definition}

\begin{theorem} The pair $(u,\om)$ obtained at the beginning of this subsection is a weak solution of the previous system.
\end{theorem}
\begin{proof} The velocity $u$ satisfies $|u|\to 0$ at infinity because the convergence of $\F^\e u^\e$ to $u$ is uniform outside a ball containing the origin, as can be checked directly by the explicit expressions for $K_\e[\om^\e]$ and $H_{\e}$, using the uniform compact support of $\om^\e$.

Moreover, using \eqref{v-phi-1}, \eqref{v-phi-2}, and $\diver H=0$, $\curl H = \d$, we obtain directly the point (b).

Next, we introduce an operator $I_\e$, which for a function $\f\in C^\infty_0([0,\infty)\times \R^2)$ gives:
\[ I_\e[\f]:= \int_0^\infty\int_{\R^2} \f_t(\F^\e)^2\om^\e dxdt+\int_0^\infty\int_{\R^2} \na\f\cdot (\F^\e u^\e)(\F^\e\om^\e) dxdt.\]
To prove that $(u,\om)$ is a weak solution, we will show that
\begin{itemize} 
\item[(i)] $I_\e[\f]+\int_{\R^2}\f(0,x)\om_0(x)dx\to 0$ as $\e\to 0$
\item[(ii)] $I_\e[\f]\to \int_0^\infty\int_{\R^2}\f_t\om dxdt +\int_0^\infty\int_{\R^2}\na\f\cdot u\om dxdt$ as $\e\to 0$.
\end{itemize}
Clearly these two steps complete the proof.

We begin by showing (i). As $u^\e$ and $\om^\e$ verify  \eqref{transport}, it can be easily seen that
\[ \int_0^\infty\int_{\R^2} \f_t(\F^\e)^2\om^\e dxdt =-\int_0^\infty\int_{\R^2} \na(\f (\F^\e)^2)\cdot u^\e\om^\e dxdt-\int_{\R^2} \f(0,x)(\F^\e)^2(x)\om_0(x) dx.\]
Thus we compute
\begin{eqnarray*}
 I_\e[\f] &=& -2\int_0^\infty\int_{\R^2} \f\na\F^\e\cdot  u^\e(\F^\e\om^\e) dxdt-\int_{\R^2} \f(0,x)(\F^\e)^2(x)\om_0(x) dx \\
 &=& -2\int_0^\infty\int_{\R^2} \f\na\F^\e\cdot  v^\e(\F^\e\om^\e) dxdt-\int_{\R^2} \f(0,x)(\F^\e)^2(x)\om_0(x) dx
 \end{eqnarray*}
 because $\na\F^\e\cdot H_{\e} =0$ (see point (i) of Lemma \ref{4.4}). By Lemma \ref{4.4} and Theorem \ref{4.2}, we have
\[ \Bigl|I_\e[\f]+\int_{\R^2} \f(0,x)(\F^\e)^2(x)\om_0(x) dx \Bigl|\leq 2 \|\F^\e\om^\e\|_{L^\infty (L^\infty)}\|\f\|_{L^1(L^\infty)}\| v^\e \|_{L^\infty(L^3)} \| \na \F^\e\|_{L^\infty(L^{3/2})}\leq C \e^{1/3},\]
which tends to zero as $\e\to 0$. This shows (i) for all $\e$ sufficiently small such that $(\F^\e)^2(x)\om_0=\om_0$, since the support of $\om_0$ does not intersect the curve.

For (ii), the linear term presents no difficulty. The second term consists of the weak-strong convergence of the pair vorticity-velocity:
\begin{eqnarray*}
\Bigl|\int\int\na\f\cdot (\F^\e u^\e)(\F^\e\om^\e)-\int\int\na\f\cdot  u\om\Bigl| &\leq& \Bigl|\int\int\na\f\cdot (\F^\e u^\e-u)(\F^\e\om^\e)\Bigl| \\
&&+\Bigl|\int\int\na\f\cdot u(\F^\e\om^\e-\om)\Bigl|.
\end{eqnarray*}
Writing $\F^\e H_\e-H=(\F^\e-1)H_\e+(H_\e-H)$ and using Theorem \ref{v-limit}, Lemmas \ref{H-limit} and \ref{4.4}, we can easily show that $\F^\e u^\e\to u$ strongly in $L^1_{\loc}(\R^+\times \R^2)$. So the first term tends to zero because $\F^\e\om^\e$ is bounded in $L^\infty(\R^+\times \R^2)$ (see \eqref{om-est-2}). In the same way, the second term tends to zero because $\F^\e\om^\e\rightharpoonup \om\text{ weak-$*$ in }L^\infty(\R^+\times \R^2)$ and $u\in L^1_{\loc}(\R^+\times \R^2)$.

Its ends the proof.
\end{proof}

Extracting again a subsequence, we can write the convergences without cutoff function. Indeed, $\|\om^\e\|_{L^\infty(\R^+\times \R^2)}$ is uniformly bounded, then we extract such that $\om^\e \rightharpoonup \om\text{, weak-$*$ in }L^\infty(\R^+\times \R^2)$. Next, for any  $T>0$ and $K$ compact set of $\R^2$, we write
\begin{eqnarray*}
 \| u^\e-u\|_{L^1([0,T]\times K)} &\leq& T \|1-\F^\e\|_{ L^2 (\R^2)} \| v^\e \|_{L^\infty ([0,T],  L^2(K))} + C_K \| \F^\e v^\e-v\|_{L^1([0,T]\times K)}\\
 &&+  T \| H_{\e}- H \|_{L^1(K)}
 \end{eqnarray*}
which tends to zero by Lemma \ref{4.4}, Theorem \ref{4.2}, Theorem \ref{v-limit} and Lemma \ref{H-limit}. Therefore, it means that $u^\e \to u$ in $L^1_{\loc}(\R^+\times \R^2)$.

Moreover,  \cite{lac_miot} establishes that the solution of \eqref{tour_equa} is unique. Although we have extracted a subsequence, we can conclude that all the sequence $(u^\e,\om^\e)$ tends to the unique pair $(u,\om)$ solution of  \eqref{tour_equa}. It ends the proof of Theorem \ref{main}.

For completeness, the reader should read Subsection 5.3 of \cite{ift_lop_euler}, concerning the asymptotic velocity equation.

\section*{Acknowledgments}

I want to thank the researchers from Nantes university for pointing out the interest of the small curve problem. I also want to warmly thank Donald E. Marshall for suggesting me \cite{war}.

\section*{Annexe}

\subsection*{Extension of Proposition \ref{2.2}}\

We prove here an extension of Proposition 2.2 from \cite{lac_euler}:

\begin{proposition}
If $\G$ is a $C^{3}$ Jordan arc, such that the intersection with the segment $[-1,1]$ is a finite union of segments and points, then there exists a unique biholomorphism $T:\Pi\to \inte\ D^c$ which verifies the following properties:
\begin{itemize}
\item $T(\infty)=\infty$ and $T'(\infty)\in \R^+_*$;
\item $T^{-1}$ and $DT^{-1}$ extend continuously up to the boundary, and $T^{-1}$ maps $S$ to $\G$;
\item $T$ extends continuously up to $\G$ with different values on each side of $\G$;
\item $D T$ extends continuously up to $\G$ with different values on each side of $\G$, except at the endpoints of the curve where $D T$ behaves like the inverse of the square root of the distance;
\item $D^2 T$ extends continuously up to $\G$ with different values on each side of $\G$, except at the endpoints of the curve where $D^2 T$ behaves like the inverse of the power $3/2$ of the distance.
\end{itemize}
\end{proposition}

We only give here the properties near the curve, because the behavior at infinity is given by Remark \ref{2.5}.

\begin{proof}
Let us work in $\C$. We follow the proof made in \cite{lac_euler}.

{\bf First step: case where $\G:= [-1,1]$.}

In the special case of the segment $[-1,1]$, we have an explicit formula of $T$, thanks to the Joukowski function 
\[G(z)= \frac{1}{2} (z+\frac{1}{z}).\]
This function maps the exterior of the disk to the exterior of the segment, and we only have to solve an equation of degree two:
\begin{equation*}
 T(z) = z\pm \sqrt{z^2-1}
\end{equation*}
where you have to choose in a good way the sign $\pm$ (see \cite{lac_euler} for more details). Hence, we have
\begin{eqnarray*}
T'(z)&=& 1\pm \frac{z}{\sqrt{z^2-1}}=1\pm \frac{z}{\sqrt{(z-1)(z+1)}}\\
T''(z)&=& \mp \frac{1}{(z^2-1)^{3/2}}=\mp \Bigl(\frac{1}{(z-1)(z+1)}\Bigl)^{3/2},
\end{eqnarray*}
which allows us to finish the proof in this case.

\bigskip

{\bf Second step: general case.}

The natural idea is to want to straighten the curve to the segment by a biholomorphism which would be $C^2$ up to the boundary. Apply after the inverse of the Joukowski function would give the result. However, it is not well established that such a straightening up exists. Of course, we know how to straighten the curve to the segment by a biholomorphism, and how to straighten up by a $C^2$ function, but we do not know how to find an application which verify the two properties (see \cite{lac_thesis} for a discussion on this subject). 

The idea in \cite{lac_euler} is to apply first the inverse of the Joukowski function. Let assume\footnote{which is possible after homothety, translation and rotation.} that the end-points of $\G$ are $-1$ and $1$, we consider the curve 
\[\tilde \G := G^{-1}(\G) = (z+\sqrt{z^2-1})(\G) \cup (z-\sqrt{z^2-1})(\G).\]
It is proved that $\tilde \G$ is a $C^{1,1}$ Jordan curve. To gain estimate of one more derivative, the only thing to do is to show that $\tilde \G$ is a $C^{2,1}$ Jordan curve.

As it is said in \cite{lac_euler}, the difficult part is to show that $\tilde \G$ is $C^{2,1}$ at the points $-1$ and $1$, where we change the sign and where the square root is non smooth. Then, let us prove it at the point $-1$.

We denote  a parametrization of the curve $\G$ by $\G(t)$ (with $\G(0)=-1$, $\G(1)=1$), and $\g_1(t)= (z+\sqrt{z^2-1})(\G(t))$, $\g_2(t)= (z-\sqrt{z^2-1})(\G(t))$.

We write the Taylor expansion of $\G(t)= -1+at + O(t^2)$ with $a\in \C$. The aim is to compute the Taylor expansion of $\frac{\g_1'(t)}{|\g_1'(t)|}$ and of $\frac{\g_2'(t)}{|\g_2'(t)|}$.

For that, we compute
\begin{eqnarray*}
\g_1(t) &=& -1+\sqrt{-2a} \sqrt{t}+at+O(t\sqrt{t})\\
\g_2(t) &=& -1-\sqrt{-2a} \sqrt{t}+at+O(t\sqrt{t})
\end{eqnarray*}
hence,
\begin{eqnarray*}
\g'_1(t) &=& \frac{\sqrt{-2a}}{2}\frac{1}{\sqrt{t}} +a+O(\sqrt{t})\\
\g'_2(t) &=& -\frac{\sqrt{-2a}}{2}\frac{1}{\sqrt{t}} +a+O(\sqrt{t}).
\end{eqnarray*}
Writing that $\dfrac{1}{|f(t)|}= \Bigl(f(t)\overline{f(t)}\Bigl)^{-1/2}$, we obtain
\begin{eqnarray*}
\frac{1}{|\g_1'(t)|} &=& \sqrt{\frac2{|a|}} \sqrt{t}- \sqrt{\frac2{|a|^3}} {\rm Re}(\overline{a}\sqrt{-2a}) t + O(t\sqrt{t}) \\
\frac{1}{|\g_2'(t)|} &=& \sqrt{\frac2{|a|}} \sqrt{t}+ \sqrt{\frac2{|a|^3}} {\rm Re}(\overline{a}\sqrt{-2a}) t + O(t\sqrt{t})  ,
\end{eqnarray*}
which give
\begin{eqnarray*}
\frac{\g_1'(t)}{|\g_1'(t)|} &=&  \frac{\sqrt{-2a}}{|\sqrt{-2a}|}  + \Bigl(a\sqrt{\frac2{|a|}}-\frac{\sqrt{-2a}}{|\sqrt{-2a}|}\frac1{|a|} {\rm Re}(\overline{a}\sqrt{-2a})\Bigl) \sqrt{t} +O(t\sqrt{t})\\
\frac{\g_2'(t)}{|\g_2'(t)|} &=& -\frac{\sqrt{-2a}}{|\sqrt{-2a}|}  + \Bigl(a\sqrt{\frac2{|a|}}-\frac{\sqrt{-2a}}{|\sqrt{-2a}|}\frac1{|a|} {\rm Re}(\overline{a}\sqrt{-2a})\Bigl) \sqrt{t} +O(t\sqrt{t}).
\end{eqnarray*}
We denote $A=a\sqrt{\frac2{|a|}}-\frac{\sqrt{-2a}}{|\sqrt{-2a}|}\frac1{|a|} {\rm Re}(\overline{a}\sqrt{-2a})$.

Let $s_1$, respectively $s_2$, the arclength coordinates associated to $\g_1$, respectively $\g_2$.
The previous computation allows us to state that
\begin{eqnarray*}
\frac{d\g_1(s)}{ds}& =&\frac{\g_1'(t)}{|\g_1'(t)|} \to  \frac{\sqrt{-2a}}{|\sqrt{-2a}|} \\
\frac{d\g_2(s)}{ds}& =&\frac{\g_2'(t)}{|\g_2'(t)|} \to - \frac{\sqrt{-2a}}{|\sqrt{-2a}|},
\end{eqnarray*}
as $t\to 0$, which means that $\tilde \G$ is $C^1$.

Moreover,
\[\frac{d^2\g_i(s)}{ds^2}= \frac{d\Bigl(  \frac{d\g_i(s)}{ds}  \Bigl)}{dt} \frac{1}{|\g_i'(t)|}= \frac{d\Bigl(  \frac{\g_i'(t)}{|\g_1'(t)|}  \Bigl)}{dt} \frac{1}{|\g_i'(t)|},\]
which implies that 
\begin{eqnarray*}
\frac{d^2\g_1(s)}{ds^2}& =&\frac{A}2\frac1{\sqrt{t}}\sqrt{\frac2{|a|}} \sqrt{t}+O(\sqrt{t})  \to \frac{A}2 \sqrt{\frac2{|a|}}\\
\frac{d^2\g_2(s)}{ds^2}& =&\frac{A}2\frac1{\sqrt{t}}\sqrt{\frac2{|a|}} \sqrt{t}+O(\sqrt{t})  \to \frac{A}2 \sqrt{\frac2{|a|}},
\end{eqnarray*}
as $t\to 0$, which means that $\tilde \G$ is $C^2$.

In the same way, we have
\[\frac{d^3\g_i(s)}{ds^3}= \frac{d\Bigl(  \frac{d^2\g_i(s)}{ds^2}  \Bigl)}{dt} \frac{1}{|\g_i'(t)|}= O\Bigl(\frac1{\sqrt{t}}\Bigl)\sqrt{\frac2{|a|}} \sqrt{t}=O(1),\]
which implies that $\frac{d^2\g_i(s)}{ds^2}$ is lipschitz in a neighborhood of $-1$.

Therefore, we have proved that $\tilde \G$ is a $C^{2,1}$ Jordan curve. Now, we can conclude as in \cite{lac_euler}.

\bigskip

For sake of clarity, we rewrite here this argument. 

We denote by $\tilde \Pi$ the unbounded connected component of $\R^2\setminus\tilde\G$. Choosing well the $\pm$, we claim that we can construct $T_2$, a biholomorphism between $\Pi$ and $\tilde \Pi$, such that $T_2^{-1}=G$. 

Next, we just have to use the Riemann mapping theorem and we find a conformal mapping $F$ between $\tilde\Pi$ and $D^c$, such that $F(\infty)=\infty$ and $F'(\infty)\in \R^+_*$. Then $T:= F\circ T_2$ maps $\Pi$ to $D^c$ and $T(\infty)=\infty$, $T'(\infty)\in \R^+_*$. To finish the proof, we use the Kellogg-Warschawski theorem (see Theorem 3.6 of \cite{pomm-2}, which can be applied for the exterior problems), to observe that $F$, $F'$ and $F''$  have a continuous extension up to the boundary, because $\tilde \G$ is a $C^{2,1}$ Jordan curve . Therefore, the behavior near the curve of $DT$ and $D^2 T$ becomes from the behavior of $T_2$ which is the inverse of the Joukowski function. Then we find the same properties as in the segment case.  

The uniqueness of $T$ can be proved thanks to the uniqueness of the Riemann mapping from $D^c$ to $D^c$ (see Remark \ref{T-unique}).
\end{proof}

\section*{List of notations}

\subsection*{Domains:}\

$D:= B(0,1)$ the unit disk and $S:= \pd D$.

$\G$ is a Jordan arc (see Proposition \ref{2.2}) and $\G_\e:= \e \G$.

$\Pi_\e := \R^2\setminus \G_\e$.

$\OM_{n}$ is a bounded, open, connected, simply connected subset of the plane, where $\pd\OM_{n}$ is a $C^\infty$ Jordan curve.

$\Pi_n:= \R^2\setminus \overline{\OM_n}$.

\subsection*{Functions:}\

$\om_0$ is the initial vorticity ($C^\infty_c(\Pi)$).

$\g$ is the circulation of $u_0^{\e}$ around $\G_{\e}$ (see Introduction).

$(u^{\e},\om^{\e})$ is the solution of the Euler equations on $\Pi_{\e}$ in the sense of Definition \ref{sol-curve}.

$T$ is a biholomorphism between $\Pi$ and $\inte\ D^c$.

$T_n$ is a biholomorphism between $\Pi_n$ and $\inte\ D^c$.

$K_\e$ and $H_\e$ are given in Subsection \ref{sect:biot}.

$K_{\e}[\om^{\e}](x):= \int_{\Pi_{\e}} K_\e(x,y) \om^{\e}(y)dy$.

$\F^{\e}$ is a cutoff function for a ${\e}$-neighborhood of $\G_{\e}$.

\end{document}